\documentclass[portrait,reqno,letterpaper, 12pt]{amsart}

 \usepackage[foot]{amsaddr}
 
\usepackage{amsmath, amsthm, amssymb, amsfonts, enumerate, tikz}
\usepackage{color}
\usepackage[numbers]{natbib}
\usepackage{dsfont}
\usepackage{graphicx}
\usepackage[pagebackref,colorlinks,urlcolor=blue,citecolor=blue,bookmarksnumbered,bookmarksopen=true,bookmarksopenlevel=0]{hyperref}
\usepackage{amsmath}
\usepackage{amsfonts}
\usepackage{amssymb}

\def\square{\hfill\hbox{\vrule height .9ex width .8ex depth -.1ex}}
\providecommand{\U}[1]{\protect\rule{.1in}{.1in}}
\newtheorem{theorem}{Theorem}
\theoremstyle{plain}

\newtheorem{corollary}{Corollary}

\newtheorem{definition}{Definition}
\newtheorem{example}{Example}

\newtheorem{lemma}{Lemma}

\newtheorem{proposition}{Proposition}
\newtheorem{remark}{Remark}

\numberwithin{equation}{section}
\newcommand{\marg}{\ensuremath{\operatorname*{marg}}}

\newcommand{\ubar}[1]{\ensuremath{\vtop{\hbox{$#1$}\kern .22ex\hrule height .06ex width .38em}}}

\def \R{\mathbb{R}}

\def \N{\mathbb{N}}

\def \1{{\bf 1}}
\newcommand{\CC}{{\mathcal{C}}}
\newcommand{\CD}{{\mathcal{D}}}

\newcommand{\CG}{{\mathcal{G}}}

\newcommand{\CU}{{\mathcal{U}}}

\newcommand{\CQ}{{\mathcal{Q}}}

\def \Min{ {\rm Min} }

\def \val{ {\rm val } }
\def \lim{ {\rm lim} }

\let\inf\relax \DeclareMathOperator*\inf{\vphantom{p}inf}
\let\max\relax \DeclareMathOperator*\max{\vphantom{p}max}
\let\min\relax \DeclareMathOperator*\min{\vphantom{p}min}

\begin{document}
\title{The Large Space of Information Structures}
\date{\today}
\author{Fabien Gensbittel}
\address{F. Gensbittel and J. Renault : Toulouse School of Economics, University Toulouse Capitole.}
\author{Marcin P{e}ski}
\address{M. P{e}ski : Department of Economics, University of Toronto.}
\author{J\'{e}r\^{o}me Renault}
 
\maketitle

\begin{abstract}
  We revisit the question of modeling incomplete information among 2 Bayesian players, following an ex-ante approach based on values of zero-sum games. 
$K$ being the finite set of possible parameters, an  information structure is defined as  a probability distribution $u$  with finite support   over $K\times \N \times \N$ with the interpretation  that:  $u$ is publicly known by the players, $(k,c,d)$ is selected according to $u$, then $c$ (resp. $d$) is  announced to player 1 (resp. player 2). Given a payoff structure $g$, composed of matrix games indexed by the state, the value of the incomplete information game defined by $u$ and $g$ is denoted $\val(u,g)$. We evaluate the pseudo-distance $d(u,v)$ between 2 information structures $u$ and $v$ by the supremum of $|\val(u,g)-\val(v,g)|$ for all $g$  with payoffs in $[-1,1]$, and  study   the metric space $Z^*$  of equivalent information structures.

We first provide a  tractable characterization of $d(u,v)$, as the minimal distance between 2 polytopes, and  recover the characterization of Peski (2008) for $u \succeq v$, generalizing to 2 players Blackwell's  comparison of experiments via garblings.  We then show that $Z^*$, endowed with a weak distance $d_W$, is homeomorphic to the  set of consistent  probabilities with finite support over the universal belief space of  Mertens and Zamir.
Finally we show the existence of   a sequence of information structures,  where players acquire more and more information,  and of $\varepsilon>0$ such that  any two elements of the sequence have distance at least $\varepsilon$  : having more and more information may lead nowhere. As a consequence, the completion of $(Z^*,d)$ is not compact, hence not  homeomorphic to the set of consistent probabilities  over the states of the world {\it \`a la} Mertens and Zamir. This  example answers by the negative the second (and last unsolved)  of the three problems posed by  J.F. Mertens in his paper ``Repeated Games",  ICM 1986.
\end{abstract}

\section{Introduction}

Given a    countable set $S$, we denote by $\Delta(S)=\{x=(x(s))_{s \in S}\in \R^S_+, \sum_{s\in S} x(s)=1\}$ the set of probability  distributions over $S$, and by $\Delta_f(S)$ the set set of probability  distributions with finite support over $S$.  The Dirac measure on an element $s$  will be denoted $\delta_s$.  More generally if  $S$ is a compact metric space, $\Delta(S)$ is the set of Borel probability distributions on $S$, and is endowed with the weak topology. 

\section{The Space of Information Structures}

Throughout the paper, $K$ is a fixed finite set of parameters  or states of nature, e.g. $K=\{0,1\}$ or $K=\{\textcolor{blue}{Blue}, \textcolor{red}{Red}\}$. There is a true state $k$ in $K$, which is imperfectly known by two Bayesian players. The general question is : What is the set of possible situations ?

\begin{definition}
An information structure is a probability with finite support over $K\times \N \times \N$.  The set of information structures is denoted by $\CU=\Delta_f(K \times \N \times \N)$.
\end{definition}

The interpretation of an information structure $u$  is the following : $u$ is publicly known by the players,   a  triple $(k,c,d)$ is selected according to $u$, then the state is $k$, player 1 learns $c$ and player 2 learns $d$.  So an information structure represents an ex-ante situation, before the players have received their signals.

Unless otherwise specified, in  our examples $K$ will have two elements and $u$ will be uniform over a finite subset of $K\times \N \times \N$.

\begin{example} \rm  $K=\{  \textcolor{blue}{Blue},   \textcolor{red}{Red}\}$, and $u$ is represented by: 
 
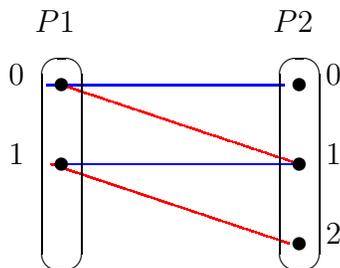
\begin{center}
 \begin{picture}(50,90)

\textcolor{blue}{\put(16,60){\line(1,0){90}}}

\textcolor{blue}{\put(15,30){\line(1,0){90}}}

\textcolor{red}{ \qbezier(10,30) (10,30)(100,00) }

\textcolor{red}{ \qbezier(10,60) (10,60)(100,30)}

\put(10, 30){\oval(15, 80)}
\put(100, 30){\oval(15, 80)}

\put(10,60){\circle*{5}}
\put(10,30){\circle*{5}}
\put(100,00){\circle*{5}}
\put(100,60){\circle*{5}}
\put(100,30){\circle*{5}}

 \put(0,80){$P1$}
  \put(90,80){$P2$}
\put(-10,60){$0$}
\put(-10,30){$1$}
\put(110,0){$2$}
\put(110,60){$0$}
\put(110,30){$1$}

\end{picture}

\end{center}
\noindent Here with probability 1/4, the top blue edge is selected, which means that the state is blue, player 1 receives the signal 0 and  player 2  receives  the signal 0. With  probability 1/4, the top red edge is selected :  the state is red, player 1 receives the signal 0 and player 2 receives the signal 1. Etc... 

After receiving signal 0, player 1 believes that both states in $K$ are equally likely. It is  the same after receiving signal 1. However the two signals of player 1 convey distinct information for him: after receiving signal 0, player 1 knows that if the state is blue  then player 2 knows it, whereas after receiving signal 1, player 1 knows that if the state is blue then player 2 has a uniform belief on $K$. \square

\end{example}

The central idea is to evaluate an information structure via   the values of    associated zero-sum Bayesian games. We first define   payoff structures, which are  given by  a matrix game with payoff in $[-1,1]$  for each state in $K$. Since we don't want to fix a priori the size of the matrices, we will formally consider  infinite matrices with only finitely many relevant rows and columns.

\begin{definition} Given $L\geq 1$, a payoff structure of size $L$ is a map $g:K\times \N\times \N\to [-1,1]$, such that for all $(k,i,j)$: $g(k,i,j)=-1$ if $i\geq L>j$ and $g(k,i,j)= 1$ if $j\geq L>i$. The set of payoff structures of size $L$ is  denoted by $\CG(L)$, and the set of payoff structures is $\CG=\bigcup_{L\geq 1} \CG(L)$.
\end{definition}

\begin{example} \label{exa2} \rm $K=\{  \textcolor{blue}{Blue},   \textcolor{red}{Red}\}$. To represent a payoff structure $g$  of size 2, it is enough to give  a blue and a red matrix such as $\left\{\textcolor{blue}{  \left(\begin{array}{cc}
 0 & 0 \\
-\frac{3}{5} & 1 \\
\end{array}
\right) \;, \; \textcolor{red}{\left(\begin{array}{cc}
 1 & -\frac{3}{5} \\
0 & 0 \\
\end{array}
\right)}}\right\}$. \square
\end{example}

\begin{definition}  An information structure $u$ and a payoff structure $g$ together define  a zero-sum Bayesian game $\Gamma(u,g)$ played as follows:
First,  $(k,c,d)$ is selected according to $u$, player 1 learns $c$ and player 2 learns $d$. Then simultaneously player 1 chooses $i$ in $\N$ and player 2 chooses $j$ in $\N$, 
 and finally the payoff of player 1 is $g(k,i,j)$.
$\Gamma(u,g)$ can be seen as a finite zero-sum game, and we denote  its value by $\val(u,g)$. 
\end{definition}

\begin{example}\label{exa3} \rm Consider the payoff structure  $g$ of example \ref{exa2}.

1)  The information structure is $u_1$, where players have complete information on the state: 

 
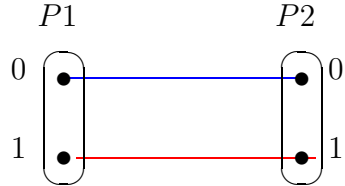
\begin{center}\begin{picture}(50,90)

\textcolor{blue}{\put(15,60){\line(1,0){90}}}

\textcolor{red}{\put(15,30){\line(1,0){90}}}

\put(10, 45){\oval(15, 50)}
\put(100, 45){\oval(15, 50)}

\put(10,60){\circle*{5}}
\put(10,30){\circle*{5}}
 
\put(100,60){\circle*{5}}
\put(100,30){\circle*{5}}

 \put(0,80){$P1$}
  \put(90,80){$P2$}

 \put(-10,60){$0$}
\put(-10,30){$1$}
 
\put(110,60){$0$}
\put(110,30){$1$}

%

\end{picture}  \end{center}
Here the unique optimal strategies are, for player 2,  to play the left column after signal 0 and to play the right column after signal 1  and, for player 1 to play the top  row  after signal 0 and the bottom row after signal 1. $\val(u_1,g)=0$.
%
%
%
%
%
%
%
%
%
%
%
%
%
%

2) The  information structure is $u_2$, with lack of information on the side of player 2 : 
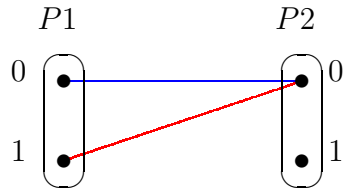
\begin{center}
 \begin{picture}(50,90)

\textcolor{blue}{\put(15,60){\line(1,0){90}}}

\textcolor{red}{ \qbezier(10,30) (10,30)(100,60) }

\put(10, 45){\oval(15, 50)}
\put(100, 45){\oval(15, 50)}

\put(10,60){\circle*{5}}
\put(10,30){\circle*{5}}
 
\put(100,60){\circle*{5}}
\put(100,30){\circle*{5}}

 \put(0,80){$P1$}
  \put(90,80){$P2$}

 \put(-10,60){$0$}
\put(-10,30){$1$}
 
\put(110,60){$0$}
\put(110,30){$1$}

%
%
%
\end{picture}

\end{center}
Here the unique optimal strategy for player 1 is to play bottom after 0 and top after 1, whereas any strategy of player 2 which plays after signal 0 both left and right with probability at least 3/8  is optimal. And $\val(u_2,g)=1/5$. Comparing with $u_1$, we recover that optimal strategies of player 1 do not only depend on his belief  on $K$.\\

3) Here the information structure  $u_3$ is given by : 

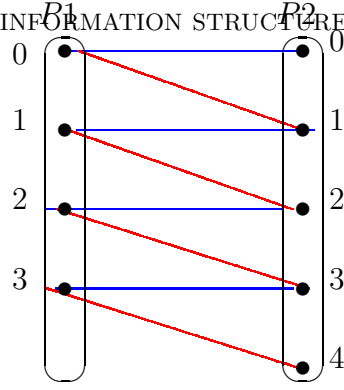
\vspace{1cm}

  \begin{center}
 \begin{picture}(50,120)

\textcolor{blue}{\put(30,60){\line(1,0){90}}}

\textcolor{blue}{\put(30,30){\line(1,0){90}}}

\textcolor{blue}{\put(30,120){\line(1,0){90}}}

\textcolor{blue}{\put(30,90){\line(1,0){90}}}

\textcolor{red}{ \qbezier(15,30) (15,30)(110,00) }

\textcolor{red}{ \qbezier(15,60) (15,60)(110,30)}

\textcolor{red}{ \qbezier(15,90) (15,90)(100,60) }

\textcolor{red}{ \qbezier(15,120) (15,120)(100,90)}

\put(10, 60){\oval(15, 130)}
\put(100, 60){\oval(15, 130)}

\put(10,60){\circle*{5}}
\put(10,30){\circle*{5}}
\put(10,90){\circle*{5}}
\put(10,120){\circle*{5}}
\put(100,90){\circle*{5}}
\put(100,120){\circle*{5}}
\put(100,00){\circle*{5}}
\put(100,60){\circle*{5}}
\put(100,30){\circle*{5}}

 \put(0,130){$P1$}
  \put(90,130){$P2$}

\put(-10,115){$0$}
\put(-10,90){$1$}
\put(-10,60){$2$}
\put(-10,30){$3$}

\put(110,0){$4$}
\put(110,60){$2$}
\put(110,30){$3$}
\put(110,90){$1$}
\put(110,120){$0$}

 
%
%
%
%

\end{picture}

\end{center}  

Here $\val(u_3,g)=1/10$, and the unique optimal strategy of player 1 is to play top after signals 0 and 2, and bottom after signal 1 and 3. Note that   player 1 should play  very differently after receiving signal 1 and 2, whereas  in both cases : player 1 believes that  both states on $K$ are equally likely, player 1 believes that player 2 believes that both states are equally likely, and player 1 believes that player  2 believes that player 1 believes that both states are equally likely. \square \end{example}

  Given $u$, $v$ in $\CU$, a natural distance between $u$ and $v$ is given by the $L^1$-norm:
$$\|u-v\|=\sum_{k\in K, (c,d)\in \N^2} |u(k,c,d)-v(k,c,d)|.$$
If  $g$ is a payoff structure in  $\CG$, since all  payoffs  are in $[-1,1]$  it is easy to see that $|\val(u,g)-\val(v,g)|\leq \|u-v\|.$ \\

We   now order and compare information structures.

\begin{definition} Given  $u$, $v$   in $\CU$, say that  $u\succeq v$  if   for all $g$ in $\CG$, $\val(u,g)\geq \val(v,g)$. \end{definition}

\begin{definition}Given  $u$, $v$   in $\CU$,
define : $$d(u,v)=\sup_{g \in \CG} |\val(u,g)-\val(v,g)|.$$
\end{definition}

Clearly $d(u,v)\leq \|u-v\|\leq 2$. $d(u,v)=d(v,u)\in [0,1]$ and  $d$ satisfies the triangular inequality but we may have $d(u,v)=0$ for $u\neq v$, so $d$ is a pseudo-distance on $\CU$. Similarly $\preceq$ is reflexive and transitive but one may have $u\succeq v$ and $v\succeq u$ for $u\neq v$. If we start  from an information structure $u$ and relabel the signals of the players, we obtain an information structure  $u'$ which is formally different from $u$, but ``equivalent" to $u$.

\begin{definition} Say that $u$ and $v$ are equivalent, and write $u\sim v$,  if for all game structures $g$ in $\CG$, $\val(u,g)= \val(v,g)$. We let $\CU^*=\CU/\sim$ be the set of equivalence classes. 
\end{definition} 

$d$ and $\succeq$ are naturally defined on $\CU^*$, and by construction $d$ is a distance and $\succeq$ is a partial order on $\CU^*$. We will  study the metric $d$, and focus on  three main questions:

1) How to compute $d(u,v)$ ?

2) What it the link between $\CU^*$  and the Mertens-Zamir space ? 

3) How large is the space of information structures ? Given $\varepsilon>0$, can we cover $\CU^*$ with finitely many balls of radius $\varepsilon$ ? \\

 Whereas previous papers in the literature restrict attention\footnote{For instance, one can read in \cite{BDM2010}  ``We leave open the question of what happens when the components of the state on which the players have some information fail to be independent.... In this situation the notion of monotonicity is unclear, and the duality method is not well understood."} to a particular subset of $\CU$ (independent information, lack of information on one side, fixed support...), we will study the general case of information structures in $\CU$ and $\CU^*$.

\section{Computing $d(u,v)$}

 We give here a tractable  characterization of $d(u,v)$,   based on duality between signals and actions. We start with  the notion of garbling, used by Blackwell to compare statistical experiments \cite{blackwell_equivalent_1953}.

\begin{definition} A garbling is an element $q: \N \to \Delta_f(\N)$, and the set of all garblings is denoted by $\CQ$. 
Given a garbling $g$ in $\CQ$  and an information structure $u$ in $\CU$, we define the information structures  $q.u$ and $u.q$ in $\CU$ by: $\forall k\in K, \forall c,c',d,d' \in \N$, $$q.u(k,c',d)=\sum_{c \in \N} u(k,c,d) q(c)(c') \;\;\; {\rm and} \; \;\; q.u(k,c,d')=\sum_{d \in \N} u(k,c,d) q(d)(d').$$ \end{definition}

The   interpretation of   $q.u$ is as follows: first $(k,c,d)$ is selected according to $u$, the state is $k$ and player 2 learns $d$. $c'$ is selected according to the probability $q(c)$, and player 1 learns $c'$. Here, the signal received by player 1 has been deteriorated through the garbling $q$. And $u.q$ corresponds to the dual situation where the signal of player 2 has been deteriorated. Since in a zero-sum game the value is monotonic in the information of the players, regardless of the payoffs player 1 always weakly  prefers $u$ to $q.u$, and $u.q$ to $u$ :  
 \begin{lemma} For all $u$ in $\CU$ and $q$ in $\CQ$, \;\; $q.u \preceq u \preceq u.q$ \end{lemma}

To compute $d(u,v)$, we will use here a second and new interpretation. A garbling $q$ in $\CQ$ will  also be seen as a behavior strategy of a player in a Bayesian game $\Gamma(u,g)$: if the signal received is $c$, play the mixed action $q(c)$. \\

\noindent {\bf Notations: } Given $L\geq 1$, we denote by $\CU(L)$ the set of information structures $u$ with support in $K\times \{0,...,L-1\}^2$ : only the first $L$ signals of each player matter.  We also  denote by $\CQ(L)$ the set of garblings $q: \N \to \Delta_f(\N)$, with range in $\Delta(\{0,...,L-1\})$.

  $\CU(L)$ is  a convex   compact subset of a finite dimensional vector space. Notice that for $u$ in $\CU$ and $L\geq 1$, the sets $\CQ(L). u=\{q.u, q\in \CQ(L)\}$ and $u.\CQ(L)=\{u.q, q\in \CQ(L)\}$ are also convex compacta in Euclidean spaces. \\
  
  Consider now $u$ and $v$ in $\CU$. Since $u$ and $v$ have finite support, we can find  $L$ such that both $u$ and $v$ are in $\CU(L)$. Our first theorem shows that  $ \sup_{g \in \CG}\left(\val(v,g)-\val(u,g) \right)$ can be simply  computed as the minimal distance, measured by the norm $\|.\|$,  between the convex compact subsets   $\CQ(L). u$ and $v.\CQ(L)$ of $\CU(L)$. Moreover, the supremum is achieved by a payoff structure of size $L$.

\begin{theorem} \label{thm1}  For $u$, $v$ in $\CU(L)$,
\begin{eqnarray*}
\sup_{g \in \CG}\left(\val(v,g)-\val(u,g) \right)&= &  \max_{g \in \CG(L)} \left(\val(v,g)-\val(u,g)\right),\\
& =&  \min_{q_1\in \CQ(L), q_2\in \CQ(L)} \| q_1.u- v.q_2 \|,\\
 &=& \min_{q_1 \in \CQ,q_2\in \CQ}\| q_1.u- v.q_2\|.
 \end{eqnarray*}
\end{theorem}

Since $d(u,v)=\max\{\sup_{g \in \CG}\left(\val(v,g)-\val(u,g) \right), \sup_{g \in \CG}$$\left(\val(u,g)-\val(v,g) \right)\}$, the following corollary is immediate, and explains how to compute $d(u,v)$. 

\begin{corollary}  \label{cor1} For $u$, $v$ in $\CU$,
$$d(u,v)  = \max_{g \in \CG} |\val(u,g)-\val(v,g)| =\max\left\{ \min_{q_1\in \CQ, q_2\in \CQ}\|q_1.u- v.q_2\|,  \min_{q_1\in \CQ, q_2\in \CQ}\|u.q_1-  q_2.v\|\right\}.$$
\end{corollary}

We can also recover from theorem \ref{thm1} that :  $ u \succeq v   \Longleftrightarrow   \exists q_1,q_2 \in \CQ, q_1.u= v.q_2$, as obtained by  Peski \cite{peski_comparison_2008}, generalizing the  Blackwell  characterization of more informative experiment   to the multi-player setting.  And we  get  a simple  characterization  of the equivalence relation:   $$u \sim v   \Longleftrightarrow \exists q_1,q_2, q_3,q_4 \in \CQ, \; q_1.u= v.q_2, u.q_3=q_4.v.$$

\vspace{0.5cm}

The proof of Theorem 1 relies on two main aspects : the two interpretations of a garbling (deterioration of signals, and strategy), and the use of a minmax theorem due to the fact that we consider information structures with finitely many signals. \\

\noindent{\bf Proof of Theorem \ref{thm1}.}  

1) We start with general considerations. For $u$ in $\CU$ and $g\in \CG$,  we denote by $\gamma_{u,g}(q_1,q_2)$ the payoff of player 1 in the zero-sum game $\Gamma(u,g)$ when player 1 plays $q_1\in \CQ$  and player 2 plays $q_2\in \CQ$. Extending as usual $g$ to mixed actions, we have: $\gamma_{u,g}(q_1,q_2)=\sum_{k,c,d} u(k,c,d) g(k, q_1(c), q_2(d)).$
  Notice that   in $\Gamma(u,g)$, both players can play the identity strategy $Id$   in $\CQ$ which   plays with probability one the signal received. And  for $u$ in $\CU$ and $g$ in $\CG$, the scalar product $\langle g,u\rangle=\sum_{k \in K, (c ,d)\in \N^2} g(k,c,d) u(k,c,d)$ is well defined, and corresponds to the expectation of $g$ with respect to $u$, and to  the payoff $\gamma_{u,g}(Id, Id)$. 

  Let us now compute the payoff $ \gamma_{u,g}(q_1, q_2)$, for any $q_1$ and $q_2$ in $\CQ$ :
%
%
%
\begin{eqnarray*}
\gamma_{u,g}(q_1, q_2) &=& \sum_{k,c,d} u(k,c,d) g(k,q_1(c),q_2(d))\\
& =&  \sum_{k,c,d} u(k,c,d) \sum_{(c',d')\in \N^2} q_1(c)(c')q_2(d)(d')g(k,c',d')\\
& =&  \sum_{k,c',d'}  g(k,c',d') \sum_{c,d} u(k,c,d)  q_1(c)(c')q_2(d)(d')\\
&=& \sum_{k,c',d'} g(k,c',d') \; q_1.u.q_2 (k,c',d')\\
&=& \langle g, q_1.u.q_2 \rangle. 
\end{eqnarray*}
 \noindent Consequently, $\val(u,g)=\max_{q_1 \in \CQ} \min_{q_2\in \CQ} \langle g, q_1.u.q_2 \rangle=  \min_{q_2\in \CQ}\max_{q_1 \in \CQ} \langle g, q_1.u.q_2 \rangle$. Since both players can play the $Id$ strategy in $\Gamma_{u,g}$, we obtain for   all  $u\in \CU(L)$ and  $g\in \CG(L)$ :
$$ \inf_{q_2\in \CQ} \langle g, u.q_2 \rangle   \; \leq \inf_{q_2\in \CQ(L)} \langle g, u.q_2 \rangle  \; \leq \; \val(u,g) \; \leq \sup_{q_1 \in \CQ(L)} \langle g, q_1.u \rangle  \;  \leq \; \sup_{q_1 \in \CQ} \langle g, q_1.u \rangle .$$

Notice also that for all $u$, $v$ in $\CU$, $\|u-v\|=\sup_{g \in \CG} \langle g, u-v \rangle$.\\

2) We now prove Theorem \ref{thm1}.

Consider $g$ in $\CG$, $q_1$ and $q_2$ in $\CQ$. $\val(v. q_2,g) \geq \val(v,g)$ and $\val(u,g)\geq \val(q_1.u,g)$, so:
$\val(v,g)-\val(u,g)\leq \val(  v.q_2, g)-\val(q_1.u,g )\leq \| q_1.u-v. q_2\|.$ We first obtain: 
\[ \sup_{g \in \CG}\left( \val(v,g)-\val(u,g)\right)\leq \inf_{q_1\in \CQ, q_2\in \CQ}\|q_1.u- v.q_2\|.\] 
Clearly,  $ \sup_{g \in \CG(L)} \left(\val(v,g)-\val(u,g)\right)\leq  \sup_{g \in \CG}\left( \val(v,g)-\val(u,g)\right)$ and $\inf_{q_1\in \CQ, q_2\in \CQ}\|q_1.u- v.q_2\| \leq \inf_{q_1 \in \CQ(L),q_2\in \CQ(L)}\|q_1.u- v.q_2\|$. So it will be enough to prove  that 
\begin{equation}\label{eq1}\inf_{q_1 \in \CQ(L),q_2\in \CQ(L)}\|q_1.u- v.q_2\| \leq  \sup_{g \in \CG(L)} \left(\val(v,g)-\val(u,g)\right).\end{equation}
We have $ \inf_{q_1 \in \CQ(L),q_2\in \CQ(L)}\|q_1.u- v.q_2\| =    \inf_{q_1 \in \CQ(L),q_2\in \CQ(L)} \sup_{g \in \CG(L)} \langle g, v.q_2 -q_1.u \rangle.$
The sets $\CQ(L)$ and $\CG(L)$ are compact,  and by Sion's theorem :
\[  \inf_{q_1 \in \CQ(L),q_2\in \CQ(L)} \sup_{g \in \CG(L)} \langle g, v.q_2 -q_1.u \rangle=  \sup_{g \in \CG(L)} \inf_{q_1 \in \CQ(L),q_2\in \CQ(L)} \langle g, v.q_2 -q_1.u \rangle.\]
Inequality  (\ref{eq1}) now follows from :   \begin{eqnarray*}  \sup_{g \in \CG(L)} \inf_{q_1 \in \CQ(L),q_2\in \CQ(L)} \langle g, v.q_2 -q_1.u \rangle &= &  \sup_{g \in \CG(L)}\left(  \inf_{q_2 \in \CQ(L)} \langle g, v.q_2 \rangle  - \sup_{q_1 \in \CQ(L)} \langle g, q_1.u \rangle \right)\\
 &\leq & \sup_{g \in \CG(L)}   \left(\val(v,g)-\val(u,g)\right).
 \end{eqnarray*}
 Finally notice that the compactness of  $\CQ(L)$ and $\CG(L)$ also give that the above infima and suprema are achieved. \square
 
 \vspace{1cm}
 
\begin{remark} \rm Theorem \ref{thm1} and its proof also imply the followings.

1) For $u$, $v$ in $\CU(L)$, the sets $A=\{q_1.u-v.q_2, q_1\in \CQ, q_2\in \CQ\}$ and $B=\{q_1.v-u.q_2, q_1\in \CQ, q_2\in \CQ\}$  are polytopes in $\R^{K\times\{0,...,L-1\}^2}$, and  to compute  $d(u,v)$ it is enough to compute $\alpha=\Min \{\|x\|_1, x\in A\}$ and $\beta=\Min \{\|x\|_1, x\in B\}$. Then $d(u,v)=\max\{\alpha, \beta\}$.\\

2)    Relationship  between $d$, $\|.\|$ and $\succeq$: We have for all $u$, $v$ in 
$\CU$,  $$ \sup_{g \in \CG}\left(\val(v,g)-\val(u,g) \right)=\min_{u' \preceq u, v'\succeq v}\|u'- v'\|.$$\\
\indent 3) Optimal payoff structure :  If $u$, $v$ are in $\CU(L)$, $\sup_{g \in \CG}\left(\val(v,g)-\val(u,g) \right)$ is achieved for $g \in \CG(L)$ maximizing $\min_{q_1,q_2\in \CQ(L) } \langle g, v.q_2 -q_1.u \rangle.$ This shows how to find $g$ such that $d(u,v)  =  |\val(u,g)-\val(v,g)|.$\\

4)  Optimal strategies : Consider $u$, $v$ in $\CU(L)$, and  let $q_1$ and $q_2$ achieving the minimum in $\min_{q'_1\in \CQ(L), q'_2\in \CQ(L)} \| q'_1.u- v.q'_2 \|$. We have  $\| q_1.u- v.q_2 \|$  $=\sup_{g \in \CG}\left(\val(v,g)-\val(u,g) \right)\leq d(u,v)$.
Let  $g$ be  a payoff structure in $\CG$, there is a canonical way to transform optimal strategies   in the Bayesian game $\Gamma(v,g)$ into $2d(u,v)$-optimal strategies in   $\Gamma(u,g)$. Indeed let $\sigma$ in $\CQ$ be optimal for player 1 in $\Gamma(v,g)$, and define $\sigma.q_1$ in $\CQ$ by $\sigma.q_1(c) =\sum_{c'} q_1(c)(c') \sigma(c')$ for each signal $c$ : player 1 receives signal $c$, then selects $c'$ according to $q_1(c)$ and plays $\sigma(c')$.   Using the notations of the proof of theorem \ref{thm1}, we have for every strategy $\tau$ of player 2 in $\CQ$:
\begin{eqnarray*}
\gamma_{u,g} (\sigma.q_1, \tau) & =& \langle g, (\sigma.q_1).u.\tau \rangle \\
 &=& \langle g,  \sigma.(q_1.u).\tau \rangle \\
 & \geq &  \langle g,  \sigma.(v.q_2).\tau \rangle - \|q_1.u-v.q_2\|\\
 &\geq & \langle g,  \sigma.v.(\tau.q_2) \rangle - d(u,v)\\\
 & \geq & \val(v,g)-   d(u,v)\\
 &\geq & \val(u,g) -2 d(u,v),
 \end{eqnarray*}
 so $\sigma.q_1$ is $2d(u,v)$ optimal in $\Gamma(u,g)$. Similarly if   $\tau$ is optimal for player 2 in $\Gamma(u,g)$, then $\tau.q_2$ is $2d(u,v)$ optimal for player 2 in $\Gamma(v,g)$. \square
\end{remark}

\begin{example}\label{exa4} \rm   Consider for instance the following information structure $u_4$.

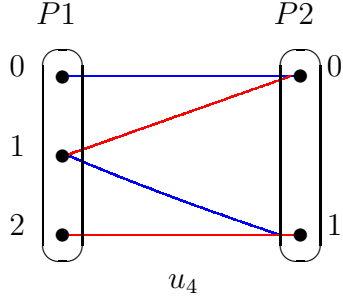
\begin{center}
  \begin{picture}(50,100)

\textcolor{blue}{ \put(20,60){\line(1,0){90}}}

\textcolor{blue}{ \qbezier(20,30)(55,15)(100,00)}

\textcolor{red}{ \qbezier(15,30) (15,30)(100,60) }

\textcolor{red}{ \put(10,00){\line(1,0){90}}}

\put(10, 30){\oval(15, 80)}
\put(100, 30){\oval(15, 80)}

\put(10,60){\circle*{5}}
\put(10,30){\circle*{5}}
\put(10,00){\circle*{5}}
\put(100,60){\circle*{5}}
\put(100,00){\circle*{5}}

 \put(0,80){$P1$}
  \put(90,80){$P2$}
\put(-10,60){$0$}
\put(-10,30){$1$}
\put(-10,00){$2$}
\put(110,60){$0$}
\put(110,00){$1$}

 \put(50,-20){$u_4$}
\end{picture}\end{center}

\vspace{0.5cm}

 How valuable is  $u_4$ to player 1, in which sense it is profitable for player 1 ? What are   $d(u_2,u_4)$  and $d(u'_2,u_4)$ ?  
 
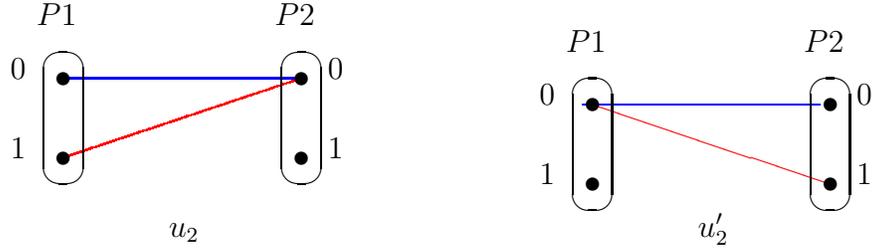
\begin{center}
 \begin{picture}(50,90)

\textcolor{blue}{\put(15,60){\line(1,0){90}}}

\textcolor{red}{ \qbezier(10,30) (10,30)(100,60) }

\put(10, 45){\oval(15, 50)}
\put(100, 45){\oval(15, 50)}

\put(10,60){\circle*{5}}
\put(10,30){\circle*{5}}
 
\put(100,60){\circle*{5}}
\put(100,30){\circle*{5}}

 \put(0,80){$P1$}
  \put(90,80){$P2$}

 \put(-10,60){$0$}
\put(-10,30){$1$}
 
\put(110,60){$0$}
\put(110,30){$1$}

 \put(50,00){$u_2$}
%
%
%
\end{picture} \hspace{5cm} 
\begin{picture}(50,100)
\textcolor{blue}{\put(10,50){\line(1,0){90}}}
\textcolor{red}{ \put(10,50){\line(3,-1){90}}}

\put(10, 35){\oval(15, 50)}
\put(100, 35){\oval(15, 50)}

\put(50,00){$u'_2$}

\put(10,50){\circle*{5}}
\put(10,20){\circle*{5}}
\put(100,50){\circle*{5}}
\put(100,20){\circle*{5}}
 
 \put(0,70){$P1$}
  \put(90,70){$P2$}
\put(-10,50){$0$}
\put(-10,20){$1$}
\put(110,50){$0$}
\put(110,20){$1$}
 
\end{picture}
\end{center}

 We first have   $\|u_2-u_4\|=1$, so $d(u_2,u_4)\leq 1$.  We have $u_2\succeq u_4$, hence $d(u_2,u_4)=\min_{q_1\in  \CQ, q_2\in \CQ}\|q_1.u_4-u_2.q_2\| .$ Define $q_1$ in $\CQ$   such that $q_10)=\delta_0$, $q_1(1)=q_1(2)=\delta_1$, and $q_2$ in $\CQ$ satisfying $q_2(0)=1/2\,  \delta_0 + 1/2 \, \delta_1$. The information structures $q_1.u_4$  and $u_2. q_2$ can be represented as follows:
 
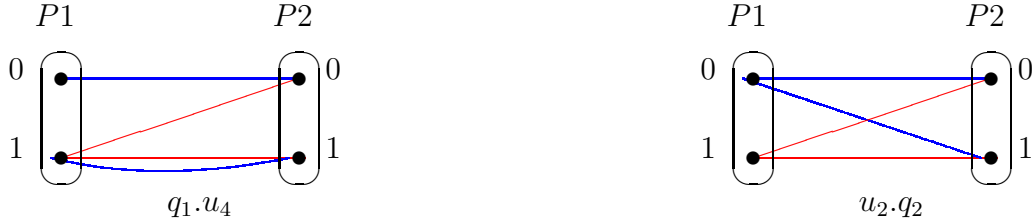
\begin{picture}(100,100)

\textcolor{blue}{\put(20,50){\line(1,0){90}}}
\textcolor{red}{ \put(20,20){\line(1,0){90}}}

\textcolor{blue}{\qbezier(10,20)(50,10)(100,20)}
 
\textcolor{red}{ \put(10,20){\line(3,1){90}}}

\put(10, 35){\oval(15, 50)}
\put(100, 35){\oval(15, 50)}

\put(50,00){$q_1.u_4$}

\put(10,50){\circle*{5}}
\put(10,20){\circle*{5}}
\put(100,50){\circle*{5}}
\put(100,20){\circle*{5}}
 
 \put(0,70){$P1$}
  \put(90,70){$P2$}
\put(-10,50){$0$}
\put(-10,20){$1$}
\put(110,50){$0$}
\put(110,20){$1$}


\textcolor{blue}{\put(270,50){\line(1,0){90}}}
\textcolor{red}{ \put(270,20){\line(1,0){90}}}

\textcolor{blue}{\qbezier(260,50)(260,50)(350,20)}
 
\textcolor{red}{ \put(260,20){\line(3,1){90}}}

\put(260, 35){\oval(15, 50)}
\put(350, 35){\oval(15, 50)}

\put(300,00){$  u_2.  q_2$}

\put(260,50){\circle*{5}}
\put(260,20){\circle*{5}}
\put(350,50){\circle*{5}}
\put(350,20){\circle*{5}}
 
 \put(250,70){$P1$}
  \put(340,70){$P2$}
\put(240,50){$0$}
\put(240,20){$1$}
\put(360,50){$0$}
\put(360,20){$1$}
\end{picture}
 
\vspace{0,5cm}

\noindent Notice that $u_2.q_2 \sim u_2$, whereas $q_1.u_4 \preceq u_4.$   $\|q_1.u_4- u_2. q_2\|=1/2$, hence $d(u_2,u_4)\leq 1/2$. 

Consider now   the payoff  structure $g$  given by  $\left\{
\textcolor{blue}{  \left(\begin{array}{cc}
 0 & 1 \\
0 & -1 \\
\end{array}
\right) }\;, \; \textcolor{red}{\left(\begin{array}{cc}
 -1 & 0 \\
1 & 0 \\
\end{array}
\right)}\right\}. $ In the game $(u_2,g)$, it is optimal for player 1 to play Top if $0$ and Bottom if $1$, and $\val(u_1,g)=1/2$. In the game $(u_4,g)$ it is optimal for player 2 to play Left if $0$ and Right if $1$, and $\val(u_4,g)=0$. Consequently, $d(u_2,u_4)\geq 1/2$, and we  obtain 
$d(u_2,u_4)=1/2.$

Notice that $u'_2\sim u''_2$, with $u''_2$ obtained from $u_2$ by exchanging the signals 0 and 1  for each player, and $\|u_4-u''_2\|=1$. Considering the payoff structure given by $\left\{
\textcolor{blue}{  \left(\begin{array}{cc}
 -1 & 1 \\
-1 & 1 \\
\end{array}
\right) }\;, \; \textcolor{red}{\left(\begin{array}{cc}
1 & -1 \\
1 & -1 \\
\end{array}
\right)}\right\}$ gives   $d(u'_2,u_4)=1$, so $u_4$ is closer to $u_2$ than to $u'_2$. \end{example}

\begin{example} \label{exa5} \rm Maximal distance  with a given marginal on $K$. Consider  $p=(p_k)_{k \in K}$ in  $\Delta(K)$. 
   $$\max\{d(u,v), \marg_{\Delta(K)} (u)=\marg_{\Delta(K)} (v)=p\}=2 \; (1- \max_k p_k).$$

\indent{\it Proof} : Assume w.l.o.g. that $p_1= \max_k p_k.$  Define $u_{max}$ and $u_{min}$  in $\CU$ such that   $u_{max}(k,c,d)=p_k \1_{c=k}\1_{d=0}$ (complete information for player 1, trivial information for player 2) and $u_{min}(k,c,d)=p_k \1_{c=0}\1_{d=k}$ for all $(k,c,d)$ (trivial information  for player 1, complete information for player 2). Since the value of a zero-sum game is weakly increasing with player 1's information and weakly decreasing with player 2's information, we have $u_{min}\preceq u\preceq u_{max}$ and $u_{min}\preceq  v\preceq u_{max}$. It implies that $d(u,v)\leq  \|u_{max}-u_{\min}\|= 2 (1-p_1)$.

Define now the payoff structure $g$ such that $g(k,c,d)=\1_{k=c} -\1_{k \neq c}.$ Clearly, $\val(u_{max},g)=1$. In   the game $\Gamma(u_{min},g)$, it is optimal for player 1 to play $c=0$, and $\val(u_{min},g)=p_1-(1-p_1)=2 p_1-1$. Hence $\val(u_{max},g)-\val(u_{min},g)= 2 (1-p_1)$, and $d(u_{\max}, u_{\min})= 2 (1-p_1)$. \end{example}

\begin{example} \label{exa6} \rm  An example of convergence in the metric space ($\CU^*,d)$ :  

\vspace{1cm}

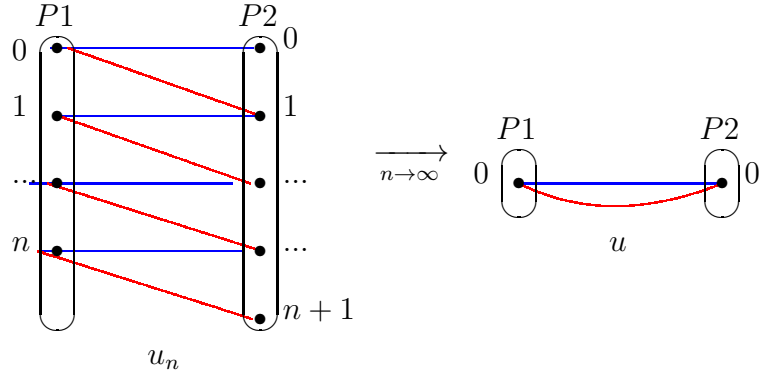
\begin{center}\setlength{\unitlength}{.3mm}
 
 \begin{picture}(350,120)

\textcolor{blue}{\put(30,60){\line(1,0){90}}}

\textcolor{blue}{\put(30,30){\line(1,0){90}}}

\textcolor{blue}{\put(30,120){\line(1,0){90}}}

\textcolor{blue}{\put(30,90){\line(1,0){90}}}

\textcolor{red}{ \qbezier(15,30) (15,30)(110,00) }

\textcolor{red}{ \qbezier(15,60) (15,60)(110,30)}

\textcolor{red}{ \qbezier(15,90) (15,90)(100,60) }

\textcolor{red}{ \qbezier(15,120) (32,114)(100,90)}

\put(10, 60){\oval(15, 130)}
\put(100, 60){\oval(15, 130)}

\put(10,60){\circle*{5}}
\put(10,30){\circle*{5}}
\put(10,90){\circle*{5}}
\put(10,120){\circle*{5}}
\put(100,90){\circle*{5}}
\put(100,120){\circle*{5}}
\put(100,00){\circle*{5}}
\put(100,60){\circle*{5}}
\put(100,30){\circle*{5}}

 \put(0,130){$P1$}
  \put(90,130){$P2$}

\put(-10,115){$0$}
\put(-10,90){$1$}
\put(-10,60){$...$}
\put(-10,30){$n$}

\put(110,0){$n+1$}
\put(110,60){$...$}
\put(110,30){$...$}
\put(110,90){$1$}
\put(110,120){$0$}

  \put(50, -20){$u_n$}
  
  \put(150,70){$ \xrightarrow[n \to \infty]{}$}
  
  \textcolor{blue}{\put(215,60){\line(1,0){90}}}
  
  \textcolor{red}{ \qbezier(210,60) (250,40)(300,60)}

    \put(250, 30){$u$}

\put(210, 60){\oval(15, 30)}
\put(300, 60){\oval(15, 30)}

\put(210,60){\circle*{5}}

\put(300,60){\circle*{5}}

 \put(200,80){$P1$}
  \put(290,80){$P2$}

 \put(190,60){$0$}

\put(310,60){$0$}

\end{picture}
\end{center}

\end{example}

\vspace{0.5cm}

The idea is that when $n$ is large, with high probability the players will receive signals  far from  0 and   $n$. These signals convey very little information to the players and only differ for very high-order beliefs. Optimal strategies of Bayesian games may   differ  after receiving one signal or another (as for $u_3$ in Example  \ref{exa3}), but if we restrict attention to   the values of the Bayesian games, $u_n$ is close to the trivial information structure $u$. 


We now prove the convergence. Consider garblings $q_1$, $q_2$, such that  $q_1(0)$ is uniform on $\{0,...,n\}$, and $q_2(c)=\delta_0$ for each $c$. Then $q_1.u=u_n.q_2$. We obtain $u\succeq u_n$, and $d(u,u_n)=\min_{q'_1,q'_2 \in \CQ} \|q'_1.u_n-u.q'_2\|$. Consider now $q'_1=q_2$ and $q'_2$ such that $q'_2(0)$ is uniform on $\{0,...,n+1\}$. We get $\|q'_1.u_n - u_nq'_2\| \leq 1/(n+1)  \xrightarrow[n \to \infty]{}0.$

\begin{remark} \label{rem2}  \rm Decision problems. Our approach can also be used for 1-player games or decision problems, with $\CU_0=\Delta_f(K \times \N)$, $\CG_0=\{g: K \times \N \to [-1,1],\; \exists L \;s.t. \forall i\geq L, \;g(k,i)=-1\}$, and $d_0(u,v)=\sup_{g \in CG_0} |\val(v,g)-\val(u,g)|$. We obtain for $u$, $v$ in $\CU_0$, that  
 $d_0(u,v)=\max \{\min_{q \in \CQ} \|q.u-v\|, \min_{q \in \CQ}\|q.v-u\|\}$ and  the Blackwell characterization : $u\succeq v \Leftrightarrow \exists q \in  Q, q.u=v$. 
 
Notice that what matters here for an information structure $u$ in $\CU_0$ is the induced law $\tilde{u}$ of the a posteriori of the player after receiving his signal. We also have, if $D$ is the set of suprema of affine  functions from $\Delta(K)$ to $[-1,1]$ and $E_1$ is the set of 1-Lipchitz  functions on $\Delta(K)$ :  
$d_0(u,v)=\sup_{f \in D}  \left|\int_{p\in \Delta(K)} f(p)d{\tilde{u}}(p) -  \int_{p\in \Delta(K)} f(p)d{\tilde{v}}(p)\right|$,  \begin{eqnarray*}
{\mathnormal{\rm and}}\;\;\; u_n  \xrightarrow[n \to \infty]{} u & \Longleftrightarrow & \forall f\in E_1, \int_{p\in \Delta(K)} f(p)d{u_n}(p)  \xrightarrow[n \to \infty]{}   \int_{p\in \Delta(K)} f(p)d{u}(p)\\
 &  \Longleftrightarrow & \left(\sup_{f \in E_1} \left( \int_{p\in \Delta(K)} f(p)d{u_n}(p) -  \int_{p\in \Delta(K)} f(p)d{u}(p)\right)  \xrightarrow[n \to \infty]{} 0\right)  \hspace{1.3cm}\square
 \end{eqnarray*} 
 

\end{remark}

\section{Links with the universal belief space}\label{sectionubs}

In the standard approach (Harsanyi, Mertens-Zamir),  a situation of incomplete information is described by a {\it state of the world}. A state of the world specifies the true state $k$, the belief of each player on $k$, the belief of each player on the belief of each player on $k$, etc... The set  of states of the world is the {\it universal belief space}  :
$$\Omega= K \times \Theta_1\times \Theta_2,$$
where for $i=1,2$, $\Theta_i$ is the {\it universal  type space} of player $i$, containing all the coherent belief hierarchies of this player. The type space of a player is  always endowed with the weak topology, and a crucial property is that $\Theta_i$ is  compact and homeomorphic to  the set   of Borel probabilities over $K \times \Theta_{-i}$. 

Any information  structure in $\CU$ naturally induces a Borel probability distribution over the universal belief space, which is consistent since we have a common prior and beliefs are derived by Bayes's rule. We denote by $\Pi$ the set of consistent (Borel) probabilities over the universal belief space, and by $\Pi_f$ the set of elements of $\Pi$ with finite support. We use the weak topology on $\Pi$ and $\Pi_f$, the space $\Pi$ is then compact and  $\Pi_f$ is dense in $\Pi$ (see corollary III.2.3 and theorem III.3.1 in \cite{mertens_sorin_zamir_2015}).    All elements of $\Pi_f$ are induced by some information structure in $\CU$, since given $P$ in 
$\Pi_f$ we can associate an   information structure $u$  in $\CU$ selecting  $(k,\theta_1, \theta_2)$ according to $P$ (formally,  $(k, f_1(\theta_1),f_2(\theta_2))$ in $K \times \N \times \N$, with  $f_1$ and $f_2$ being one-to-one). 

Given $P$ in $\Pi_f$ and $g$ in $\CG$, we can define $\val(P,g)$ as the value of the zero-sum Bayesian game where first: $(k, \theta_1, \theta_2)$ is selected according to $P$, then the players simultaneously select  $i$ and $j$ in $\N$, and the payoff to player 1 is $g(k,i,j)$. By Proposition III.4.4 in \cite{mertens_sorin_zamir_2015}, $\val(u,g)=\val(\Phi(u),g)$ and an optimal strategy in the game defined by $P$ and  $g$ induces an optimal strategy in the zero-sum game $\Gamma(u,g)$. Now, it is known that the value functions of finite games separate the elements of $\Pi$ (lemma 41 in Gossner Mertens \cite{gossner_value_2001}),  so   equivalent information structures in $\CU$ induce  the same element of $\Pi_f$, and we can associate to each equivalence class  in  $\CU^*$ an element  of $\Pi_f$. We obtain a natural bijection  from $\CU^*$ to $\Pi_f$, that we denote by $\Phi$, and one can ask  how similar the topological spaces $\CU^*$ and $\Pi_f$ are. \\

 In this section only, we will not  consider the distance $d$, but  the weak topology of pointwise convergence on $\CU$ and $\CU^*$. 
 
 \begin{definition}
 A sequence of information structures $(u_n)_{n\geq 1}$ weakly converges  to $u$ if for all payoff structures $g$ in $\CG$, $\val(u_n,g)\xrightarrow[n \to \infty]{}\val(u,g)$.
  \end{definition}
  Since the set of payoff structures can be seen as a  countable union of sets  of payoff matrices of a given size, one can find a sequence $g_1$,...,$g_n$,... of elements of $\CG$ 
 such that for each $g$ in $\CG$ and $\varepsilon>0$, there exists $n$ with $\max_{k\in K,(i,j)\in \N^2} |g(k,i,j)-g_n(k,i,j)|\leq \varepsilon$. The sequence $(g_n)$ is dense in $\CG$ for the sup norm, and  the weak convergence is metrizable by the metric: 
 $$d_W(u,v)=\sum_{n=1}^\infty \frac{1}{2^n} \left|val(u,g_n)-\val(v,g_n\right)|.$$
 $(\CU^*,d_W)$ is now another metric space, a priori different from $(\CU^*,d)$ since we have changed the metric. It can not be compact, since we have only considered information structures with finite support. 
 
 \begin{theorem} \label{thm2}  $\;$
 
1)  The metric space  $(\CU^*,d_W)$ is homeomorphic  to the space $\Pi_f$  of consistent   probabilities with finite support over the universal belief space.
 
 2) Its completion is homeomorphic to the compact space $\Pi$ of consistent   probabilities  over the universal belief space.

 \end{theorem}
 
 \noindent{\bf Proof of Theorem \ref{thm2}.}   
 
  Define, for $P$ and $Q$ in $\Pi$, $$d^*_W(P,Q)=\sum_{n=1}^\infty \frac{1}{2^n} \left|\val(P,g_n)-\val(Q,g_n)\right|.$$
 If for each $n$,  $\val(P,g_n)-\val(Q,g_n)=0$ then for all $g$ in $\CG$, $\val(P,g_n)-\val(Q,g_n)=0$,  and $P=Q$ by lemma 41 of   \cite{gossner_value_2001} again. $d^*_W$ is a metric on $\Pi$. 

 For each payoff structure $g$, the mapping $(P\mapsto \val(P,g))$ is continuous for the weak topology on $\Pi$ (see Lemma 2 in  \cite{mertens_repeated_1986} or  Proposition III.4.3. in \cite{mertens_sorin_zamir_2015}). So if a sequence $(P_t)_t$ of elements of $\Pi$ weakly converges to some limit $P$, we have $d^*_W(P_t,P)\xrightarrow[t\to \infty]{} 0$.
 
  Conversely, consider a sequence $(P_t)_t$ of elements of $\Pi$ converging for $d^*$ to some limit $P$, we have for all  $n$  : $\val(P_t,g_n)\xrightarrow[t\to \infty]{} \val(P,g_n)$. For any converging subsequence $(P_{\phi_t})_t$, for  the weak topology, with limit $Q$, we have by the previous paragraph, that for all $n$, $\val(P_{\phi_t},g_n)\xrightarrow[t\to \infty]{} \val(Q,g_n)$. So $d^*_W(P,Q)=0$ for each limit point $Q$, and since  $\Pi$ is compact the sequence $(P_t)_t$ converges to $P$. 
  
  We obtain that $d^*_W$ induces the weak topology on $\Pi$. By construction,  the  bijection $\Phi$ is isometric from $(\CU^*, d_W)$ to $(\Pi_f,d^*_W)$, hence an homeomorphism.
 
Finally, the completion of $(\CU^*, d_W)$  is homeomorphic to the completion of $(\Pi_f,d^*_W)$. Since $d^*_W$ induces the weak topology on $\Pi$, the completion of $(\Pi_f,d^*_W)$ is the closure of $\Pi_f$. Since  $\Pi$ is compact and $\Pi_f$ is dense in $\Pi$, this completion is $\Pi$. \square \\
 
Theorem 2 suggests a possible alternative construction of the set $\Pi$ of consistent probability over the universal belief space. The alternative construction is simply based on the values of finite zero-sum Bayesian games.  \\

In the remainder of the paper we  come back to the distance $d$ on $\CU^*$.

\section{How large is the space of information structures ?}

We consider the metric space $(\CU^*,d)$ (or simply $\CU^*$). As $\CU$ only contains information structures with finite support, $\CU^*$ can not be compact, 
and we denote by $\overline{\CU}$ its completion. We focus here on a major property : is $\overline{\CU}$ compact ? 
Equivalently, is $\CU^*$ totally bounded, i.e. given $\varepsilon>0$ can we cover $\CU^*$ with finitely many balls of radius $\varepsilon$ ? Can we see $\CU^*$ as a subset of a compact metric space ?

One can show that this question is equivalent to any of the following ones:

 A)  Is $\overline{\CU}$ homeomorphic to the set $\Pi$ of consistent probabilities over the universal belief space ?

B)  Are the distances $d$ and the weak distance $d_W$   uniformly equivalent on  $\CU^*$?

C)   Is the family  $(P \mapsto \val(P,g))_{g \in CG}$ an  equicontinous family of mappings from $\Pi$ to $\R$ ? \\

Question  C) corresponds to the second of the three  problems\footnote{Problem 1 asked for   the convergence of the value functions $(v_\lambda)_\lambda$ and $(v_n)_n$ in a  general zero-sum repeated game  with finitely many states, actions and signals, and was disproved during  the PhD thesis of B. Ziliotto  \cite{ziliotto_zero-sum_2016}. Problem 3 asks if the existence of a uniform value follows  from  the uniform convergence of $(v_\lambda)$, and was disproved by Lehrer and Monderer \cite{lehrer_discounting_1994}  for 1-player games, see also \cite{monderer_asymptotic_1993}.}
posed by J.F. Mertens in his Repeated Games survey from ICM 1986  \cite{mertens_repeated_1986} : ``This equicontinuity or Lispchitz property character is crucial in many papers...".

\begin{remark} \rm  Repeated Games. Consider a general   zero-sum repeated game (stochastic game, with incomplete information and signals), given by a  transition $q:K\times I\times J\longrightarrow \Delta(K \times C\times D)$,  a payoff function $g:K\times I\times J\longrightarrow [-1,1]$ and an initial probability $u_0$ in $\Delta(K \times A \times B)$, where $K$, $I$, $J$, $A$ and $B$ are finite subsets of $\N$.    Before  stage 1, an initial state $k_1$ in $K$ and initial private signals $a_1$ in $A$ for player 1, and $b_1$ in  $B$ for player 2, are selected  according to $u_0$.  Then at   each stage $t$,   simultaneously player 1 chooses an action $i_t$ in $I$ and player 2 chooses and action $j-t$ in $J$, and  :  the stage payoff is $g(k_t,i_t,j_t)$, an element $(k_{t+1},a_{t+1},b_{t+1})$ is selected according to $g(k_t,i_t,j_t)$, the new state is $k_{t+1}$, player 1 receives the signal $a_{t+1}$, player 2 the signal $b_{t+1}$, and the play proceeds to stage $t+1$. 
 
An appropriate state variable is here $u$ in $\CU$, representing the current state in $K$ and the finite sequence of signals previously received by each player. As a consequence, a recursive formula  can be explicitly written as follows:   for all discount $\lambda$ in $(0,1]$ and all $u$ in $\CU$, 
\begin{eqnarray*}v_\lambda(u) &=& \max_{q_1\in \CQ}  \min_{q_2\in \CQ}\;  \lambda  G(u,q_1,q_2)+(1-\lambda)  v_\lambda(F(u,q_1,q_2)),\\
&=&\min_{q_2\in \CQ} \max_{q_1\in \CQ}\;  \lambda G(u,q_1,q_2)+(1-\lambda) v_\lambda(F(u,q_1,q_2)),
\end{eqnarray*}
with $ G(u,q_1,q_2)=\sum_{k,c,d} u(k,c,d) g(k,q_1(c),q_2(d)) \in [-1,1]$, 
and $F(u,q_1,q_2)\in \CU$ is defined, for all $(k,i,a,j,b)$ in $K \times I \times A\times J\times B$, by  $F(u,q_1,q_2)(k', f_1(c,i,a),f_2(d,j,b))=$ $\sum_{k} u(k,c,d) q_1(c)(i)q_2(d)(j) q(k,i,j)(k',a,b)$ (where $f_1$ and $f_2$ are fixed one-to-one mappings from $\N^3$ to $\N$).

The value function $v_\lambda$ can be approximated by the value functions  of finite games. Since such value functions are, by construction, 1-Lipschitz   from $(\CU,d)$ to $[-1,1]$, so is $v_\lambda$.  Hence   the family $(v_\lambda)_\lambda$ is equicontinuous, and if it happens that the set of  information  structures that can be reached during the game  is totally bounded, by Ascoli's theorem this family has  a uniform limit point when $\lambda \to 0$. \square
 \end{remark}

Compactness of $\overline{\CU}$ is then strongly related to the equivalence between the strong distance $d$ and the weak distance $d_W$. Notice that in the 1-player case  of Remark \ref{rem2}, weak and strong convergence are equivalent, and ${\CU}_0$ is homeomorphic to   $\Delta_f(\Delta(K))$, which is dense in the compact set $\Delta(\Delta(K))$.
For 2 players,  compactness has been obtained in every particular case  tackled  so far. If $\CU'$ is a subset of $\CU^*$, we denote by $\overline{\CU'}$ the closure of $\CU'$ in  $\overline{\CU}$ : 

$\bullet$ Set  $\CU_1$ of  information structures where both players receive  the same signal:  $\overline{\CU_1}$ is compact, and homeomorphic to $\Delta(\Delta(K))$.  
  Here given $u$ in $\CU_1$, what matters is the induced law $\tilde{u}$ on the common a posteriori of the players on $K$. Another characterization of $d(u,v)$ has been obtained in  \cite{renault-venel}.  let $D_1$ be the subset of 1-Lipschitz functions from $\Delta(K)$ to $\R$ satisfying $ \forall p,q \in \Delta(K), \forall a, b \geq 0, \; a f(p)-b f(q)\leq \| a p-b q\|_1$. We  have :
    $$\forall u,v \in \CU_1,\;d(u,v)= \sup_{f \in D_1}  \left( \int_{p\in \Delta(K)} f(p)d{\tilde{u}}(p) -  \int_{p\in \Delta(K)} f(p)d{\tilde{v}}(p)\right).$$

$\bullet$ Set $\CU_2$   of information structures where player 1 knows the signal of player 2: $\overline{\CU_2}$ is compact, and homeomorphic to $\Delta(\Delta(\Delta(K)))$ (see \cite{mertens_repeated_1986}, \cite{gensbittel2014}). 

$\bullet$ Set $\CU_3$ of independent information structures : $\CU_3$ is the set of $u$ in $\CU$ such that $u(c,d | k)=u(c |k) u(d | k)$ (the signals $c$ and $d$ are conditionally independent given $k$). Here  $\overline{\CU_3}$ is homeomorphic to $\Delta(\Delta(K)\times \Delta(L)).$\\

We now present our main counterexample, where it is assumed that there are at least 2 states in $K$.

\begin{theorem}   \label{thm3}

There exists $\varepsilon>0$ and a sequence $(\mu^l)_{l \geq 1}$ of information   structures  in $\CU$  satisfying :

  {\it 1)} $d(\mu^l, \mu^p)>\varepsilon$ for all $l\neq p$,

 {\it 2)}  for each  $l$ the conditional law of  $\mu^{l+1}$ on the support of $\mu^l$ is $\mu^l$, and

  {\it 3)} for  all  $l>p$, the distribution on states and $2p$-order beliefs induced by $\mu^l$ does not depend on $l$.

\end{theorem}

\noindent{\bf Remarks : }

Condition {\it 1)}  implies that $(\CU^*,d)$ is not totally bounded, 
and $\overline{\CU}$ is not compact. The space  of information structures $\CU^*$  is very large, in the sense that it is   {\it not} a subset of a compact metric space, one cannot approximate the space with finite sets.  All questions A), B), C) above have a negative answer, in particular $\overline{\CU}$ is not homeomorphic to $\Pi$.  
 
Condition {\it 2)} means that to  go from $\mu^l$ to $\mu^{l+1}$, each player gets an extra signal. So having more and more information may lead... nowhere. This has to be contrasted with the 1-player case, where the sequence of beliefs of a player receiving more and more signals is a martingale, which converges  in law. We don't have a ``strategic martingale" convergence theorem here.

Condition {\it 3)}  implies there exists no $n$ such that knowing the joint distribution of $n$-order beliefs is enough to determine, up to $\varepsilon$, the value of every  finite  game with payoffs in $[-1,1]$.

Computing the largest $\varepsilon$ such that a sequence satisfying  condition  {\it 1)} exists seems very difficult, but we believe it is very small.  Rough estimates of our proof only gives $\varepsilon \geq 3. 10^{-17}$. \\

\section{Proof of theorem \ref{thm3}.}
Without loss of generality we assume that there are two states: $K=\{0,1\}$. For convenience we will consider information structures $u$ in $\Delta(K \times \CC \times \CD)$ where $\CC$ and $\CD$ are arbitrary finite sets (which can be easily identified with subsets of $\N$). Similarly, we will consider game structures $g: K \times \CC \times \CD\to [-1,1]$, where 
$\CC$ and $\CD$ are the respective finite sets of actions of player 1 and player 2. 

 $N$ is  a very large even integer to be fixed later, and we   write $A=C=D=\{1,...,N\}$, with the idea of using $C$ while speaking of actions or signals of player 1, and using $D$ while speaking of actions and signals of player 2. We fix $\varepsilon$ and $\alpha$, to be used later,  such that $$0<\varepsilon<\frac{1}{10(N+1)^2}, \; {\rm and }\; \alpha=\frac{1}{25}.$$

We will consider a   Markov chain with law $\nu$ on $A$, satisfying: 

$\bullet$  the law of the first state of the Markov chain is uniform on $A$,

$\bullet$  for each  $a$ in $A$, there are exactly  $N/2$ elements $b$ in $A$ such that $\nu(b|a)=2/N$ : given that the current state of the  Markov chain is $a$, the law of the next state is uniform on a subset of states of size $N/2$, 

$\bullet$ and two more conditions, called  $UI1$ and $UI2$, to be be defined later. 

A sequence $(a_1,...,a_l)$ of length $l\geq 1$ is said to be {\it nice}  it it is in the support of the Markov chain: $\nu(a_1,...,a_l)>0$. For instance  any sequence of length 1 is nice, and $N^2/2$ sequences of length 2 are nice.  The proof is now split in 3 parts: we first define the information structures $(u^l)_{l\geq 1}$ and some  payoff  structures $(g^p)_{p\geq 1}$. Then we define the conditions $UI1$ and $UI2$ and show that they imply the conclusions of theorem \ref{thm3}. Finally, we show, via the probabilistic method,  the existence of a Markov chain $\nu$ satisfying all  our  conditions.

\subsection{Information and payoff  structures $(u^l)_{l\geq 1}$ and $(g^p)_{p\geq 1}$.}

\begin{definition} For $l\geq 1$, define the information structure $u^l\in \Delta(K \times C^l \times D^l)$   by: for each state $k$ in $K$, signal  $c=(c_1,...,c_l)$ in $C^l$ of player 1 and signal $d=(d_1,...,d_l)$ in $D^l$  for player 2, 
\begin{equation*}
u^{l}(k,c,d)= \nu(c_1,d_1,c_2,d_2,...,c_l,d_l)  \left( \frac{c_{1}}{N+1}  \mathbf{
1}_{k=1}+\frac{1-c_{1}}{N+1}   \mathbf{1}_{k=0} \right).
\end{equation*}
\end{definition}
\noindent The following interpretation of $u^l$ holds: first select $(a_1,a_2,...,a_{2l})=(c_1,d_1,...,c_l,d_l)$ in $A^{2l}$ according to the Markov chain $\nu$ (i.e. uniformly among the nice sequences of length $2l$), then tell  $(c_1,c_2,...,c_l)$ (the elements of the sequence with odd indices) to player 1, and $(d_1,d_2,...,d_l)$  (the elements of the sequence with even  indices)  to player 2. Finally choose the state $k=1$ with probability $c_1/(N+1)$, and state $k=0$ with  the complement probability $1-c_1/(N+1)$.

Notice that the definition is not symmetric among players,  the first signal $c_1$ of player 1 is uniformly distributed and plays a particular role.   The marginal of $u^l$ on $K$ is uniform, and the marginal of $u^{l+1}$ over $(K \times C^l \times V^l)$ holds : condition ${\it 2)}$ of   theorem \ref{thm3} is satisfied.

We now show that condition ${\it 3)}$ of the theorem holds. Recall that $n $-order
beliefs are defined inductively as conditional laws. Precisely, the first
order beliefs $\theta _{1}^{i}$ of player $i$ is the conditional law of $%
k$ given her signal. The $n$-order belief $\theta _{n}^{i}$ of player $%
i$ is the conditional law of $(\omega ,\theta _{n-1}^{-i})$ given her
signal. In this construction, conditional laws are seen as random variables
taking values in space of probability measures.

\begin{lemma}
\label{lem_beliefs} For all $l>p$, the joint distribution of $%
(\omega,\theta^1_{2p},\theta^2_{2p})$ induced by the information structure $u^l$  is
independent of $l$.
\end{lemma}

\begin{proof}
We use the notation ${%
\mathcal{L}}(X|Y)$ for the conditional law of $X$ given $Y$,  and the
identification $(a_{1},...,a_{2l})=(c_{1},d_{1},....,c_{l},d_{l})$. At first, note that by construction $k$ and $(a_2,....,a_{2l})$ are
conditionally independent given $a_1$, so that the sequence $%
(k,a_1,a_2,...,a_{2l})$ is a Markov process. It follows that 
$\theta^1_1 = {\mathcal{L}}( k | c_1,...,c_l)$ =$ {\mathcal{L}}(k
|c_1).$
The Markov property implies that 
\begin{equation*}
\theta^2_1={\mathcal{L}}(k| d_1,....,d_l)= {\mathcal{L}}(k|d_1),
\; \theta^2_2={\mathcal{L}}(d, \theta^1_1(c_1) | d_1,....,d_l)= {%
\mathcal{L}}(k, \theta^1_1(c_1)|d_1),
\end{equation*}
and therefore we have 
$$
\theta^1_2= {\mathcal{L}}(k, \theta^2_1(d_1) | c_1,....,c_l)= {\mathcal{%
L}}(k, \theta^2_1(d_1)|c_1,c_2).
$$
By induction, and applying the same argument (future and past of a Markov
process are conditionally independent given the current position), we deduce
that for all $n\geq 1$, 
\begin{equation*}
\theta^1_{2n}= {\mathcal{L}}(k,\theta^{2}_{2n-1}|
c_1,....,c_{\min(l,n+1)}), \;\;\;
\theta^1_{2n+1}= {\mathcal{L}}(k,\theta^{2}_{2n}|
c_1,....,c_{\min(l,n+1)}),
\end{equation*}
\begin{equation*}
\theta^2_{2n-1}= {\mathcal{L}}(k,\theta^{1}_{2n-2}|
d_1,....,d_{\min(l,n)}), \;\;\;
\theta^2_{2n}= {\mathcal{L}}(k,\theta^{1}_{2n-1}|
d_1,....,d_{\min(l,n)}).
\end{equation*}
As a consequence, for all $n\leq p$, these conditional laws do not depend
on which $u^l$  we are using as soon as $l>p$.\end{proof}

 Let us give already a very rough intuition of the conditions $UI1$ and $UI2$ and the Bayesian games that we will consider. The players will be asked to report their signals, and payoffs will highly depend on whether the reported sequence is nice or not. And, thanks to the conditions $UI1$ and $UI2$,  the chain will be such that if $(c_1,d_1,...,c_l,d_l)$ is selected according to $\nu$ and player 2 only knows $(d_1,...,d_l)$, any deviation of player 2  to some $(d_1,...,d_{r-1},d_r,...,d'_l)$, with $d'_r\neq d_r$,    will satisfy:
 
\begin{center}\begin{tabular}{lcc}
$\nu\left( (c_1,d_1,...,c_r,d'_r)\; {\mathnormal {\rm is}} \;  {\mathnormal  {\rm nice}} \right)$ &$\simeq$ &1/2,\\
$ \nu\left( (c_1,d_1,...,c_r,d'_r,c_{r+1})\; {\mathnormal {\rm is}} \;  {\mathnormal  {\rm nice}} \right) $&$\simeq$& 1/4,\\
 $  \nu\left( (c_1,d_1,...,c_r,d'_r,c_{r+1},d'_{r+1})\; {\mathnormal {\rm is}} \;  {\mathnormal  {\rm nice}} \right)$ &$\simeq$ &1/8,
   \end{tabular}\end{center}
   
  etc...,   and similar conditions for deviations of  player 1. 
  
    \begin{definition}  
Consider  a sequence $(a_1,...,a_l)$ of elements of $A$ which is not nice,  i.e. such that   $\nu(a_1,...,a_l)=0$. We say that the sequence is not nice because of player 1 if $\min\{t\in \{1,...,l\}, \nu(a_1,...,a_t)=0\}$ is odd, and not nice because of player 2 if $\min\{t\in \{1,...,l\}, \nu(a_1,...,a_t)=0\}$ is even. 
      \end{definition}
      A sequence $(a_1,...,a_l)$ is now either nice, or not nice because  of player 1, or not nice because of player 2. A  sequence of length  2 is either nice, or not nice because of player 2. 
      
  \begin{definition} For $p\geq 1$, define the payoff structure $g^p: K \times C^p \times D^{p-1}\to [-1,1]$ such that  for all $k$ in $K$, $c'=(c'_1,...,c'_p)$ in $C^p$, 
  $d'=(d'_1,...,d'_{p-1})$ in $D^{p-1}$ : 
\begin{eqnarray*} g^p(k,c',d')& =&g_0(k,c'_1)+h^p(c',d'), \;\;{\mathnormal with}\\
g_0(k,c'_1) & = &  - {\left( k- \frac{u'_1}{N+1}\right)}^2 + \frac{N+2}{6(N+1)}, \;   \\
   h^p(c',d')& =&\left\{ \begin{array} {ccl} 
\varepsilon & \mbox{if} & (c'_1,d'_1,..., c'_p) \;{is} \; {\it nice},\\
   5\varepsilon &  \mbox{if} &  (c'_1,d'_1,..., c'_p) \;{is}\; not \;  {\it nice} \;  because\;  of \; player \;2,   \\
-5  \varepsilon    &  \mbox{if} & (c'_1,d'_1,..., c'_p)\;{is} \; not \;  {\it nice}  \;  because\;  of \; player \;1.
    \end{array}\right.
\end{eqnarray*}
    \end{definition}
 
One can check that $|g^p|\leq 5/6+5 \varepsilon\leq 8/9$. Regarding the $g_0$ part of the payoff, consider a decision problem for player 1 where: $c_1$ is selected uniformly in $A$ and   the state is selected to be  $k=1$  with probability $c_1/(N+1)$ and $k=0$  with probability $1-c_1/(N+1)$. Player 1  observes $c_1$ but not $k$, and he choose $c'_1$ in $A$ and receive  payoff $g_0(k,c'_1)$.  We have $\frac{c_1}{N+1} g_0(1,c'_1)+ (1-\frac{c_1}{N+1}) g_0(0,c'_1)$ $=$ $ \frac{1}{(N+1)^2}( c'_1(2c_1-c'_1)+ (N+1)((N+2)/6-c_1))$.  To maximize this  expected payoff, it is well known that player 1 should play his belief on $k$, i.e. $c'_1=c_1$. Moreover,   if player 1 chooses $c'_1\neq c_1$, its expected loss from not having chosen $c_1$ is at least $\frac{1}{(N+1)^2}\geq 10 \varepsilon$. And the constant $\frac{N+2}{6(N+1)}$ has been chosen such that the value of this decision problem is 0.

    Consider now $l\geq 1$ and $p\geq 1$. By definition, the Bayesian game $\Gamma(u^k,g^p)$ is played as follows: first, $(c_1,d_1,...,c_l,d_l)$ is selected according to the law $\nu$ of the Markov chain, player 1 learns $(c_1,...,c_l)$, player 2 learns $(d_1,...,d_l)$   and the state is  $k=1$ with probability $c_1/(N+1)$ and $k=0$  otherwise. Then {\it simultaneously}  player 1 chooses  $c'$ in $C^p$ and player 2 chooses $d'$ in $D^{p-1}$, and finally the payoff to player 1 is $g^p(k,c',d')$. Notice that by the previous paragraph about $g_0$, it is always strictly  dominant for player 1 to report correctly his first signal, i.e. to choose $c'_1=c_1$. We will show in the next section that  if $l\geq p$ and player 1   simply plays the sequence of signals he received, player 2 can not do better than also reporting truthfully  his own signals, leading to a value not lower than the payoff for nice sequences, that is  $\varepsilon$. On the contrary in the game $\Gamma(u^l,g^{l+1})$, player 1   has to report not only the $l$ signals he has received, but also an extra-signal $c'_{l+1}$ that he has to guess. In this game we will prove that  if player 2   truthfully reports  his own signals, player 1 will incur the payoff $-5\varepsilon$ with probability at least (approximately) 1/2, and this will result in a low value. These intuitions will prove correct in the next section, under some conditions $UI1$ and $UI2$.
    
    \subsection{Conditions UI and values} 
    
    To prove that the intuitions of the previous paragraph are correct, we need to ensure that  players
have incentives to report their true signals, so we need additional assumptions on the Markov chain. \\

\noindent{\bf Notations and definition:} Let $l\geq 1$, $m\geq 0$, $c=(c_1,...,c_l)$ in $C^l$ and $d=(d_1,...,d_m)$ in $D^m$. We write :

\centerline{\begin{tabular}{lcll}
$a^{2q}(c,d)$ & = &$(c_{1},d_{1},....,c_{q},d_{q})\in A^{2q}$  &{for each }
$q\leq \min \{l,m\}$, \\
$a^{2q+1}(c,d)$& = &$(c_{1},d_{1},...., c_{q},d_{q},c_{q+1})\in A^{2q+1} $ & { for
each }$q\leq \min \{l-1,m\}$.
\end{tabular}}

\noindent For $r\leq \min \{2l,2m+1\}$, 

\centerline{we say that $c$  and $d$ are \emph{nice
at level }$r$, and we write $c\smile _{r}d,$ if $a^{r}(c,d)$ is nice.}  


\vspace{0,5cm}

In the next definition we consider an information structure $u^l\in \Delta(K \times C^l\times D^l)$ and denote by ${\tilde c}$ and ${\tilde d}$ the respective random variables of the signals of player 1 and 2. 

\begin{definition}$\;$ \label{defUI}

We say that the conditions $UI1$ are satisfied if  for all $l\geq 1$,  
all ${c}=({c}_{1},...,{c}_{l})$ in $C^{l}$ and $c^{\prime
}=(c_{1}^{\prime },...,c_{l+1}^{\prime })$ in $C^{l+1}$  such that $%
{c}_{1}=c_{1}^{\prime }$, we have 
\begin{equation} \label{eq61}
  u^l \left( c^{\prime }\smile _{2l+1}\tilde{d}\;  \big| \; \tilde{c}={c},c^{\prime
}\smile _{2l}\tilde{d}\right)  \in [1/2-\alpha,1/2+\alpha]
\end{equation}%
and  for all $m\in \{1,...,l\}$ such that ${c}%
_{m}\neq c_{m}^{\prime }$, for $r=2m-2,2m-1$,
\begin{equation}\label{eq62}
 u^l \left(c^{\prime }\smile _{r+1}\tilde{d} \; \big| \; \tilde{c}=c,c^{\prime
}\smile _{r} \tilde{d}\right)  \in [1/2-\alpha,1/2+\alpha].
\end{equation}%

We say that the conditions $UI2$ are satisfied if  for all $1\leq p\leq l$,
for all ${d}\in D^{l}$, for all $d^{\prime }\in D^{p-1}$, for all $m\in
\{1,...,p-1\}$ such that ${d}_{m}\neq d_{m}^{\prime }$, for $r=2m-1,2m$ 
\begin{equation}\label{eq63}
u^l\left( \tilde{c}\smile _{r+1}d^{\prime }|\tilde{d}={d},\tilde{c}\smile
_{r}d^{\prime }\right) \in [1/2-\alpha,1/2+\alpha].
\end{equation}
\end{definition}
To understand the conditions $UI1$, consider the Bayesian game $\Gamma(u^l, g^{l+1})$, and assume that player 2 truthfully reports his sequence of signals and  that player 1 has received the signals $(c_1,...,c_l)$ in $C^l$.   (\ref{eq61}) states that if  the sequence of reported signals $(c'_1,\tilde{d}_1,...,c'_l,\tilde{d}_l)$ is nice at level $2l$, then whatever the last reported signal $c'_{l+1}$, the conditional  probability that $(c'_1,\tilde{d}_1,...,c'_l,\tilde{d}_l, c'_{l+1})$ is still nice is in $[1/2-\alpha, 1/2 + \alpha]$, i.e.  close to 1/2. Regarding  (\ref{eq62}), first notice that if $c'=c$, then by construction  $(c'_1,\tilde{d}_1,...,c'_l,\tilde{d}_l)$ is nice and $ u^l \left( c^{\prime }\smile _{r+1}\tilde{d} \; \big| \; \tilde{c}=c,c^{\prime
}\smile _{r} \tilde{d}\right) = u^l \left( c \smile _{r+1}\tilde{d} \; \big| \; \tilde{c}=c\right)=1$ for each $r=1,...,2l-1$. Assume now that for some $m=1,...,l$, player 1 misreports his $m^{th}$-signal, i.e.  reports $c'_m\neq c_m$. (\ref{eq62}) requires that given that the reported signals were nice so far (at level $2m-2$), the conditional probability that the reported signals are not nice at level $2m-1$ (integrating $c'_m$)  is close to 1/2, and moreover if the reported signals are nice at this level $2m-1$, adding the next signal $\tilde{d}_{m}$ of player 2 has probability close to 1/2 to keep the reported sequence nice. Conditions $UI2$  have a similar interpretation, considering the Bayesian games $\Gamma(u^l,g^p)$ for $p\leq l$, assuming that player 1 reports truthfully his signals and that player 2 plays $d'$ after having received the signals $d$. 

\begin{proposition} \label{pro2} Conditions $UI1$ and $UI2$ imply : 
\begin{eqnarray} \forall l\geq 1, \forall p\in \{1,...,l\}, & \val(u^l,g^p)\geq \varepsilon. \label{eq64} \\
\forall l \geq 1,  &  \val(u^l,g^{l+1})\leq -\varepsilon. \label{eq66} \end{eqnarray}
\end{proposition}

As a   consequence of this proposition, under conditions $UI1$ and $UI2$ we easily obtain  condition  $1)$ of theorem \ref{thm3} : 
\begin{corollary}  \label{cor2} If $l\neq p$ then $d(u^l,u^p)\geq 2\varepsilon.$
\end{corollary}
\proof Assume $l>p$, then  $d(u^l,u^p)\geq \val(u^l,g^{p+1})-\val(u^p,g^{p+1})\geq \varepsilon - (-\varepsilon).$ \square

\vspace{0.5cm}

\noindent{\bf Proof of proposition \ref{pro2}.} We assume that $UI1$ and $UI2$ hold. We fix $l\geq 1$,  work on the probability space $K \times C^l \times D^l$ equipped
with the probability $u^l$, and denote by $\tilde{c}$ and $\tilde{d}$ the random variables of the signals received by the players.

1) We first prove (\ref{eq64}), and consider the game $\Gamma(u^l,g^p)$ with $p\in \{1,...,l\}$.
We assume that player 1 chooses the truthful strategy. Fix  $d=(d_1,...,d_l)$ in $D^l$ and $d'=(d'_1,...,d'_{p-1})$ in $D^{p-1}$, and assume that player 2 has received the signal  $d$  and  chooses to report $d'$.



Define the non-increasing sequence of events:  $$A_n=\{\tilde{c} \smile_n d'\}.$$
We will prove by backward induction  that:
\begin{equation}\label{eq65}\forall n=1,...,p, \;\;  \mathbb{E}[ h^p(\tilde{c},d') |\tilde{d}=d, A_{2n-1} ] \geq
\varepsilon. \end{equation}
 
If $n=p$, $h^p(\tilde{c},d')=\varepsilon$ on the event $A_{2p-1}$, implying
the result. Assume now that for some $n$ such that $1\leq n<p$, we have :
$ \mathbb{E}[h^{p}(\tilde{c},d')|\tilde{d}=d, A_{2n+1}]\geq
\varepsilon.$
 Since we have a non-increasing sequence of events, 
$
\mathds{1}_{A_{2n-1}}=\mathds{1}_{A_{2n+1}}+\mathds{1}_{A_{2n-1}}\mathds{1}%
_{A_{2n}^{c}}+\mathds{1}_{A_{2n}}\mathds{1}_{A_{2n+1}^{c}},$
so by definition of the payoffs,  
$h^{p}(\tilde{c},d')\mathds{1}_{A_{2n-1}}=h^{p}(\tilde{c},d')\mathds{1}%
_{A_{2n+1}}+5\varepsilon \mathds{1}_{A_{2n-1}}\mathds{1}_{A_{2n}^{c}}-5%
\varepsilon \mathds{1}_{A_{2n}}\mathds{1}_{A_{2n+1}^{c}}.
$

First assume  that $d'_n=d_n$. By  construction  of the Markov chain, $u^l(A_{2n+1}|A_{2n-1},\tilde{d}=d)=1$,  implying that 
$u^l(A_{2n+1}^{c}|A_{2n-1},\tilde{d}=d)=u^l(A_{2n}^{c}|A_{2n-1},\tilde{d}=d)=0$. As a consequence,
\begin{eqnarray*}
\mathbb{E}[h^{p}(\tilde{c},d')|\tilde{d}=d,A_{2n-1}]& =&\mathbb{E}[h^{p}(\tilde{c},d')\mathds{1}_{A_{2n+1}} |\tilde{d}=d,A_{2n-1}] \\
& = &\mathbb{E}[  \mathbb{E}[h^{p}(\tilde{c},d')|\tilde{d}=d, A_{2n+1}]\mathds{1}_{A_{2n+1}} |\tilde{d}=d,A_{2n-1}]   \\
& \geq \varepsilon .
\end{eqnarray*}%
Assume now that  $d'_n\neq d_n$. Assumption UI2 implies that :
\begin{eqnarray*}
 u^l(A_{2n}^{c}|A_{2n-1},\tilde{d}=d) &\geq &1/2-\alpha,\\
 u^l(A_{2n}\cap A_{2n+1}^{c}|A_{2n-1},\tilde{d}=d) &\leq& (1/2+\alpha)^2,\\
  u^l(A_{2n+1}|A_{2n-1},\tilde{d}=d) & \geq &(1/2-\alpha)^2.
\end{eqnarray*}%
It follows that : 
\begin{align*}
\mathbb{E}[h^{p}(\tilde{c},d'|\tilde{d})=d,A_{2n-1}]& =\mathbb{E}[\mathbb{E}[h^{p}(\tilde{c},d')|\tilde{d}=d,A_{2n+1}]%
\mathds{1}_{A_{2n+1}}|\tilde{d}=d,A_{2n-1}] \\
& \quad +5\varepsilon u^l(A_{2n}^{c}|A_{2n-1},\tilde{d}=d)-5\varepsilon
u^{l}(A_{2n}\cap A_{2n+1}^{c}|A_{2n-1},\tilde{d}=d) \\
& \geq \varepsilon \,  (\frac{1}{4}-\alpha +\alpha ^{2})+5\, \varepsilon \,  (\frac{1}{2%
}-\alpha )-5\, \varepsilon \,  (\frac{1}{4}+\alpha +\alpha ^{2}) \\
& =\varepsilon \,  (\frac{3}{2}-11 \alpha - 4 \alpha^2)\geq \varepsilon ,
\end{align*}%
 And (\ref{eq65})  follows by backward induction.
 
Since $A_1$ is an event which holds almost surely, we deduce that 
$\mathbb{E}[ h^p(\tilde{c},d' )|\tilde{d}=d]  \geq \varepsilon. 
$ Hence   the truthful strategy of player 1 guarantees the payoff $\varepsilon$ in $\Gamma(u^l,g^p)$.\\

2)  We now prove (\ref{eq66}) and consider the  Bayesian game $\Gamma(u^l,g^{l+1})$, assuming 
 that player 2  chooses the truthful strategy.  Fix  $c=(c_1,...,c_l)$ in $C^l$ and $c'=(c'_1,...,c'_{l-1})$ in $C^{l-1}$, and assume that player 1 has received the signal  $c$   and  chooses to report $c'$.  We will show that the expected payoff of player 1    is not larger than  $-\varepsilon$, and assume w.l.o.g. that $c'_1=c_1$. 
 Consider the
non-increasing sequence of events : $$B_n=\{c^{\prime }\smile_n \tilde{d}\, \}.$$ We will prove by backward induction 
that:  
\begin{equation*}
\forall n=1,...,l, \;\;  \mathbb{E}[ h^{l+1}(c^{\prime },\tilde{d}) | \tilde{c}=c, B_{2n} ] \leq
-\varepsilon .
\end{equation*}

If $n=l$, we have $\mathds{1}_{B_{2l}}= \mathds{1}_{B_{2l+1}}+ \mathds{1}%
_{B_{2l}}\mathds{1}_{B_{2l+1}^c}$, and 
$
h^{l+1}(c^{\prime },\tilde{d})\mathds{1}_{B_{2l}}= \varepsilon\mathds{1}%
_{B_{2l+1}}-5\varepsilon \mathds{1}_{B_{2l}}\mathds{1}_{B_{2l+1}^c}.
$
UI1 implies that $
|u^l( B_{2l+1} | \tilde{c}=c, B_{2l} ) - \frac{1}{2} | \leq \alpha$
, and it  follows that :
\begin{align*}
\mathbb{E}[h^{l+1}(c^{\prime },\tilde{d})|\tilde{c}=c,B_{2l}]& =\varepsilon \, 
u^{l}(B_{2l+1}|\tilde{c}=c,B_{2l})-5\varepsilon \, u^{l}(B_{2l+1}^{c}|u=
\hat u,B_{2l}) \\
& \leq \varepsilon \, (\frac{1}{2}+\alpha )-5\varepsilon \, (\frac{1}{2}-\alpha ) \leq -\varepsilon. \end{align*}

Assume now that   for some $n=1,..., l-1$, we have $
\mathbb{E}[h^{l+1}(c^{\prime },\tilde{d})|\tilde{c}=c,B_{2n+2}]\leq
-\varepsilon.$
We have $
\mathds{1}_{B_{2n}}=\mathds{1}_{B_{2n+2}}+\mathds{1}_{B_{2n}}\mathds{1}%
_{B_{2n+1}^{c}}+\mathds{1}_{B_{2n+1}}\mathds{1}_{B_{2n+2}^{c}},$
 and by definition of $h^{l+1}$, 
\begin{equation*}
h^{l+1}(c^{\prime },\tilde{d})\mathds{1}_{B_{2n}}=h^{l+1}(c^{\prime },\tilde{d})\mathds{1}%
_{B_{2n+2}}-5\varepsilon \mathds{1}_{B_{2n}}\mathds{1}_{B_{2n+1}^{c}}+5%
\varepsilon \mathds{1}_{B_{2n+1}}\mathds{1}_{B_{2n+2}^{c}}.
\end{equation*}%

First assume that $c_{n+1}^{\prime }= c_{n+1}$, then $u^{l}(B_{2n+2}|B_{2n},\tilde{c}=c)=1$.  
Then :
\begin{eqnarray*}
\mathbb{E}[h^{l+1}(c^{\prime },\tilde{d})|\tilde{c}=c,B_{2n}]& =&\mathbb{E}[h^{l+1}(c^{\prime },\tilde{d}) \mathds{1}_{B_{2n+2}}|\tilde{c}=c,B_{2n}], \\
&=& \mathbb{E}[\mathbb{E}[h^{l+1}(c^{\prime },\tilde{d})|\tilde{c}=c,B_{2n+2}]%
\mathds{1}_{B_{2n+2}}|\tilde{c}=c,B_{2n}]  \leq -\varepsilon.
\end{eqnarray*}

Assume on the contrary  that  $c_{n+1}^{\prime }\neq c_{n+1}$, assumption UI1 implies that :
\begin{eqnarray*}
 u^{l}(B_{2n+1}^{c}|B_{2n},\tilde{c}=c) & \geq &1/2-\alpha,\\
 u^{l}(B_{2n+1}\cap B_{2n+2}^{c}|B_{2n},\tilde{c}=c)&\leq& (1/2+\alpha)^2,\\
 u^{l}(B_{2n+2}|B_{2n},\tilde{c}=c)  &\geq &(1/2-\alpha)^2.
\end{eqnarray*}%
It follows that :
\begin{align*}
\mathbb{E}[h^{l+1}(c^{\prime },\tilde{d})|\tilde{c}=c,B_{2n}]& =\mathbb{E}[\mathbb{E}[h^{l+1}(c^{\prime },\tilde{d})|\tilde{c}=c,B_{2n+2}]%
\mathds{1}_{B_{2n+2}}|\tilde{c}=c,B_{2n}] \\
& \quad -5\, \varepsilon \,  u^{l}(B_{2n+1}^{c}|B_{2n},\tilde{c}=c)+5 \, \varepsilon \, 
u^{l}(B_{2n+1}\cap B_{2n+2}^{c}|B_{2n},\tilde{c}=c) \\
& \leq - \, \varepsilon \,  (\frac{1}{4}-\alpha+\alpha ^{2} )-5 \, \varepsilon \, (\frac{1}{%
2}-\alpha )+5 \, \varepsilon \,  (\frac{1}{4}+\alpha+\alpha^2 ) \leq -\varepsilon.
\end{align*}%

By induction, we obtain $ \mathbb{E}[ h^{l+1}(c^{\prime },\tilde{d}) | \tilde{c}=c, B_{2} ] \leq
-\varepsilon$. Since $B_2$  holds almost surely here, we get
$\mathbb{E}[ h^{l+1}(c^{\prime },\tilde{d}) |\tilde{c}=c]  \leq -\varepsilon,$
  showing  that the truthful strategy of player 2  guarantees that the payoff of the maximizer is less or equal to $%
-\varepsilon$, and concluding the proof. \square

\subsection{Existence of an appropriate Markov chain} \label{existence}

Here we conclude the proof of Theorem \ref{thm3} by showing the existence of an even integer $N$ and a Markov chain with law $\nu$ on $A=\{1,...,N\}$ satisfying our conditions :

$1)$  the law of the first state of the Markov chain is uniform on $A$,

$2)$  for each  $a$ in $A$, there are exactly  $N/2$ elements $b$ in $A$ such that $\nu(b|a)=2/N$, 

$3)$    $UI1$ and $UI2$. 

Denoting by  $P=(P_{a,b})_{(a,b)\in A^2}$   the transition matrix of the Markov chain, we have to  prove the existence of $P$ satisfying $2)$ and $3)$. The proof is non constructive and uses the following probabilistic method, where we select independently for each $a$ in $A$, the set $\{b\in A, P_{a,b}>0\}$ uniformly among the subsets of $A$ with cardinal $N/2$.  We will show that when $N$ goes to infinity,  the probability of selecting  an appropriate transition matrix does not only become positive, but converges to 1. 

Formally, denote by 
$\mathcal{S}_{A} $ the collection of all subsets $S\subseteq A$ with cardinality $%
\left\vert S\right\vert =\frac{1}{2}N$.  We consider  
a collection  $\left( S_{a}\right) _{a\in A}$ of     i.i.d. random variables   uniform distributed  over $%
\mathcal{S}_{A}$ defined on a probability space $(\Omega_N,\mathcal{F}_N,%
\mathbb{P}_N)$. For all $a$, $b$ in $A$, let $$X_{a,b}=\mathds{1}_{\{ b\in S_{a}\}} \; {\rm and}\; 
P_{a,b}=\frac{2}{N}X_{a,b}.$$
By construction, $P$ is a transition matrix satisfying  $2)$. Theorem \ref{thm3} will now follow directly from the following proposition. 
\begin{proposition} \label{pro3}   
$$\mathbb{P}_N  \left(\text{ }P\text{ induces a Markov chain  satisfying 
 UI1 and UI2 }\right)  \xrightarrow[n \to \infty]{}1.$$ In particular,   the
above probability is strictly positive for all sufficiently large $N$.
\end{proposition}
The rest of this section is devoted to the proof of proposition \ref{pro3}.

%
%
%
%
We start with  probability bounds based on Hoeffding's inequality.

\begin{lemma}
\label{lemT1}For any $a\neq b,$ each $\gamma >0$ 
\begin{equation*}
\mathbb{P}_N\left(\left\vert |S_{a}\cap S_{b}|-\frac{1}{4}N\right\vert \geq
\gamma N \right)\leq \frac{1}{2} e^{4}Ne^{-2\gamma^2 N}.
\end{equation*}
\end{lemma}
\begin{proof}
Consider a family of i.i.d. Bernoulli variables $(\widetilde
X_{i,j})_{i=a,b, \, j \in A}$ of parameter $\frac{1}{2}$ defined on a space $%
(\Omega,\mathcal{F},\mathbb{P})$. For $i=a,b$, define the events $\widetilde
L_i=\{ \sum_{j \in A} \widetilde X_{i,j} = \frac{N}{2}\}$ and the set-valued
variables $\widetilde S_i = \{ j \in A \,|\, \widetilde X_{i,j}=1 \}$. It is
straightforward to check that the conditional law of $(\widetilde S_a,
\widetilde S_b)$ given $\widetilde L_a \cap \widetilde L_b$ under $\mathbb{P}
$ is the same as the law of $(S_a,S_b)$ under $\mathbb{P}_N$. It follows
that 
\begin{align*}
\mathbb{P}_N\left(\left\vert |S_{a}\cap S_{b}|-\frac{1}{4}N\right\vert \geq
\gamma N \right)&=\mathbb{P}\left(\left\vert |\widetilde S_{a}\cap\widetilde
S_{b}|-\frac{1}{4}N\right\vert \geq \gamma N \,\Big|\, \widetilde L_a\cap
\widetilde L_b \right) \\
&\leq \frac{\mathbb{P}\left(\left\vert |\widetilde S_{a}\cap\widetilde
S_{b}|-\frac{1}{4}N\right\vert \geq \gamma N \right)}{\mathbb{P}\left(
\widetilde L_a\cap \widetilde L_b \right)}.
\end{align*}
Using Hoeffding inequality, we have 
$$
\mathbb{P}\left(\left\vert |\widetilde S_{a}\cap\widetilde S_{b}|-\frac{1}{4}%
N\right\vert \geq \gamma N \right) =\mathbb{P}\left(\left\vert \sum_{j \in
A} \widetilde X_{a,j} \widetilde X_{b,j}-\frac{1}{4}N\right\vert \geq \gamma
N \right)  \leq 2 e^{-2 \gamma^2 N}.
$$
On the other hand, using Stirling approximation\footnote{We have $n^{n+\frac{1}{2}%
}e^{-n} \leq n! \leq e n^{n+ \frac{1}{2}} e^{-n}$ for each $n$.}, we have 
\begin{equation*}
\mathbb{P}\left( \widetilde L_a\cap \widetilde L_b \right)= \left( \frac{1}{%
2^N} \frac{ N !}{ \left(\frac{N}{2}! \right)^2} \right)^2 \geq \left( \frac{%
2^{N+1} N^{- \frac{1}{2}}}{2^N e^2} \right)^2 = \frac{4}{N e^4}.
\end{equation*}
We deduce that 
$
\mathbb{P}_N\left(\left\vert |S_{a}\cap S_{b}|-\frac{1}{4}N\right\vert \geq
\gamma N \right)\leq \frac{1}{2} e^{4}Ne^{-2\gamma^2 N}.
$
\end{proof}

\begin{lemma}
\label{lemT2}For each $a\neq b,$ for any subset $S\subseteq A$ and any $\gamma \geq \frac{1}{2N-2}$, %
$$
\mathbb{P}_N\left( \left\vert \sum_{i\in S}X_{i,a}-\frac{1}{2}\left\vert
S\right\vert \right\vert \geq \gamma N\right)  \leq  2e^{-2 N \gamma ^{2}}, \;{\mathnormal and }\; 
\mathbb{P}_N\left( \left\vert \sum_{i\in S}X_{i,a}X_{i,b}-\frac{1}{4}%
\left\vert S\right\vert \right\vert \geq \gamma N\right)  \leq  2e^{- \frac{1}{2} N \gamma ^{2}}.$$
%
\end{lemma}

\begin{proof}
For the first inequality, notice that $X_{i,a}$ are i.i.d. Bernoulli random  variables with parameter $\frac{1}{2}$. The Hoeffding inequality implies that :
\begin{equation*}
\mathbb{P}_N\left( \left\vert \sum_{i\in S}X_{i,a}-\frac{1}{2}\left\vert
S\right\vert \right\vert \geq \gamma N\right) \leq  2e^{-2 \gamma ^{2}\frac{N^2}{|S|}}\leq 2e^{-2 N \gamma ^{2}}.
\end{equation*}

For the second inequality, let $Z_{i}=X_{i,a}X_{i,b}.$ Notice that all
variables $Z_{i}$ are i.i.d. Bernoulli random  variables with parameter  $p=\frac{1}{2}%
\left( \frac{\frac{N}{2}-1}{N-1}\right) =\frac{1}{4}-\frac{1}{4N-4}$. The Hoeffding inequality implies that 
\begin{eqnarray*}
\mathbb{P}_N\left( \left\vert \sum_{i\in S}Z_{i}-\frac{1}{4}\left\vert
S\right\vert \right\vert \geq \gamma N\right) &\leq &\mathbb{P}_N\left(
\left\vert \sum_{i\in S}Z_{i}-p\left\vert S\right\vert \right\vert \geq 
\frac{1}{2}\gamma N\right) \\
&\leq & 2e^{-2 \gamma ^{2}\frac{N^2}{|S|}}\leq 2e^{-\frac{1}{2} N \gamma ^{2}},
\end{eqnarray*}
where we used that $|S||p-\frac{1}{4}| \leq \frac{N}{4N-4} \leq \frac{\gamma
N}{2}$ for the first inequality. \end{proof}

\vspace{0.3cm}

\begin{definition}
For each $a\neq b$ and $c\neq d,$ each $\gamma >0,$ define :

 \begin{tabular}{lll}
$Y_{a}  =  2\sum_{i\in A}X_{i,a},$ &   $Y^{c}=2\sum_{i\in A}X_{c,i}=N$, & \;\\
$Y_{a,b}  = 4\sum_{i\in A}X_{i,a}X_{i,b},$ & $ Y_{a}^{c}=4\sum_{i\in A}X_{i,a}X_{c,i}, $ & $ Y^{c,d}=4\sum_{i\in A}X_{c,i}X_{d,i},$ \\
$Y_{a,b}^{c} = 8\sum_{i\in A}X_{i,a}X_{i,b}X_{c,i},$ & $ Y_{a}^{c,d}=8\sum_{i\in A}X_{i,a}X_{c,i}X_{d,i},$ & $Y_{a,b}^{c,d} = 16\sum_{i\in A}X_{i,a}X_{i,b}X_{c,i}X_{d,i}.$
\end{tabular}
\end{definition}

\vspace{0.3cm}

\begin{lemma}
\label{lemT3}For each $a\neq b$ and $c\neq d,$ each $\gamma \geq 64/N,$ each
of the variables $%
Z\in\{Y_{a},Y^{c},Y_{a,b},Y^{c,d},Y_{a}^{c},Y_{a,b}^{c},Y_{a}^{c,d},Y_{a,b}^{c,d}\}, 
$ 
\begin{equation*}
\mathbb{P}_N\left( \left\vert Z-N\right\vert \geq \gamma N\right) \leq
e^4Ne^{-\frac{N}{32}{( \frac{\gamma}{10})}^2}.
\end{equation*}
\end{lemma}

\begin{proof}
In case $Z=Y_{a} $ or $Y_{a,b},$ the bound follows from Lemma \ref{lemT2} (for $%
S=A) $. If case $Z=Y^{c},$ the bound is trivially  satisfied. If $Z=Y^{c,d},$ the bound follows from Lemma \ref{lemT1}.

In case $Z=Y_{a,b}^{c,d}$, notice that 
$$
Y_{a,b}^{c,d}= 16\sum\limits_{i\in S_{c}\cap S_{d}}Z_{i},\;\; {\rm where }\;\;  Z_{i}=X_{i,a}X_{i,b}.$$
 All variables $Z_{i}$ are i.i.d. Bernouilli random variables with parameter 
 $p=%
\frac{1}{4}-\frac{1}{4N-4}$.  Moreover, $%
\left\{ Z_{i}\right\} _{i\neq c,d}$ are independent of $S_{c}\cap S_{d}$. Up
to enlarge the probability space, we can construct a new collection of i.i.d.
Bernoulli  random  variables $Z_{i}^{\prime }$ such that $Z_{i}^{\prime
}=Z_{i}$ for all $i\neq c,d$ and such that $\left\{ (Z_{i}^{\prime
})_{i \in A},S_{c}\cap S_{d}\right\} $ are all independent. Then, 
\begin{equation*}
\left\vert Y_{a,b}^{c,d}-16\sum\limits_{i\in S_{c}\cap S_{d}}Z_{i}^{\prime
}\right\vert \leq 32,
\end{equation*}%
and, because $\frac{1}{2}\gamma N\geq 32,$ we have%
\begin{equation*}
\mathbb{P}_N\left( \left\vert Y_{a,b}^{c,d}-N\right\vert \geq \gamma
N\right) \leq \mathbb{P}_N\left( \left\vert \sum\limits_{i\in S_{c}\cap
S_{d}}Z_{i}^{\prime }-\frac{1}{16}N\right\vert \geq \frac{1}{32}\gamma
N\right) .
\end{equation*}%
Define the events 
$$
A =\left\{ \left\vert \frac{1}{4} \left\vert S_{c}\cap S_{d}\right\vert -\frac{N}{16}%
\right\vert \geq \frac{1}{160}\gamma N\right\} ,  \;\;\;
B =\left\{ \left\vert \sum\limits_{i\in S_{c}\cap S_{d}}Z_{i}^{\prime }-%
\frac{1}{4}\left\vert S_{c}\cap S_{d}\right\vert \right\vert \geq \frac{1}{40%
}\gamma N\right\} .$$
 
Then, the probability can be further bounded by 
\begin{equation*}
 \leq \mathbb{P}_N\left( A\right) +\mathbb{P}_N\left( B \right) \leq 
\frac{1}{2} e^{4}Ne^{-2N\left( \frac{1}{40}\gamma \right) ^{2}}+2e^{-\frac{1}{2}N\left( \frac{1}{40}\gamma \right) ^{2}}\leq e^4Ne^{-\frac{N \gamma^2}{3200}}
\end{equation*}%
where the first bound comes from Lemma \ref{lemT1}, and the second from the
second bound in Lemma \ref{lemT2}.

The remaining bounds have proofs similar (and simpler)  to the  case $Z=Y_{a,b}^{c,d}$.
%
%
%
\end{proof}

\vspace{0.3cm}

Finally, we describe an event $E$ that collects these bounds. Recall that  $\alpha=1/25$, and define 
 for each $a\neq b$ \ and $%
c\neq d$, 
\begin{eqnarray*}
E_{a,b,c,d} &=&\left\{ \left\vert \frac{Y_{a,b}}{Y_{a}}-1\right\vert \leq
2\alpha \right\} \cap \left\{ \left\vert \frac{Y_{a,b}^{c}}{Y_{a}^{c}}%
-1\right\vert \leq 2\alpha \right\} \cap \left\{ \left\vert \frac{Y_{a}^{c,d}%
}{Y_{a}^{c}}-1\right\vert \leq 2\alpha \right\} \cap \left\{ \left\vert \frac{%
Y_{a,b}^{c,d}}{Y_{a}^{c,d}}-1\right\vert \leq 2\alpha \right\} \\
&&\left\{ \left\vert \frac{Y^{c,d}}{Y^{c}}-1\right\vert \leq 2\alpha \right\}
\cap \left\{ \left\vert \frac{Y_{a}^{c}}{Y^{c}}-1\right\vert \leq 2\alpha
\right\} \cap \left\{ \left\vert \frac{Y_{a}^{c,d}}{Y^{c,d}}-1\right\vert
\leq 2\alpha \right\} .
\end{eqnarray*}%
Finally, let 
\begin{equation*}
E=\bigcap\limits_{a,b,c,d:a\neq b\ \text{and }c\neq d}E_{a,b,c,d}.
\end{equation*}

\vspace{0.3cm}

\begin{lemma}\label{lembound}
We have 
\begin{equation*}
\mathbb{P}_N(E)> 1- 7  e^{4}N^5e^{-%
\frac{N}{2163200}} \xrightarrow[n \to \infty]{}1.
\end{equation*}%
\end{lemma}

\begin{proof}
Take $\gamma =\frac{\alpha }{1+\alpha}=\frac{1}{26}$ and let 
\begin{equation*}
F_{a,b,c,d}=\bigcap%
\limits_{Z\in \{Y_{a},Y_{a,b},Y^{c,d},Y^{c,d},Y_{a}^{c},Y_{a,b}^{c},Y_{a}^{c,d},Y_{a,b}^{c,d}\}}\left\{ \left\vert Z-N\right\vert \leq \gamma N\right\} .
\end{equation*}%
It is easy to see that $F_{a,b,c,d}\subseteq E_{a,b,c,d}.$ The probability
that $F_{a,b,c,d}$ holds can be bounded from Lemma \ref{lemT3} (as soon as $N\geq \frac{64}{\gamma} =1664$), as%
\begin{equation*}
\mathbb{P}_N\left( F_{a,b,c,d}\right) \geq 1-7  e^{4}Ne^{-%
\frac{N}{32.(260)^2}}.
\end{equation*}%
The result follows since  there are less  than   $N^{4}$ ways of
choosing $(a,b,c,d)$.
\end{proof}

 \vspace{0.5cm}
Computations using the bound of lemma \ref{lembound} show that $N= 52. 10^6$ is enough to have the existence of an appropriate Markov chain. So one can take $\varepsilon=3. 10^{-17}$ in the statement of theorem \ref{thm3}. 
We conclude the proof of proposition \ref{pro3} by showing  that event $E$ implies conditions $UI1$ and $UI2.$

\begin{lemma}
\label{lemT4} If event $E$ holds, then the conditions $UI1,UI2$ are satisfied.
\end{lemma}

\begin{proof}
We fix  the law $\nu$ of the Markov chain on $A$  and assume that it has been  induced, as explained  at the beginning of section \ref{existence}, by a transition matrix $P$
 satisfying $E$. For $l\geq 1$, we forget  about the state in $K$  and  still denote by $u^l$ the marginal of $u^l$ over $C^l\times D^l$. If  $c=(c_1,...,c_l) \in C^l$ and $d=(d_1,...,d_l)\in D^l$, we have $u^l(c,d)=\nu (c_1,d_1,...,c_l,d_l)$. \\

Let us begin with condition UI2 which we recall here:  for all $1\leq p\leq l$,
for all ${d}\in D^{l}$, for all $d^{\prime }\in D^{p-1}$, for all $m\in
\{1,...,p-1\}$ such that ${d}_{m}\neq d_{m}^{\prime }$, for $r=2m-1,2m$,
\begin{equation*}
u^l\left( \tilde{c}\smile _{r+1}d^{\prime }|\tilde{d}={d},\tilde{c}\smile
_{r}d^{\prime }\right) \in [1/2-\alpha,1/2+\alpha], \;\; \hspace{3cm}(\ref{eq63})
\end{equation*}
where  $(\tilde{c}, \tilde{d})$ is a random variable selected according  to $u^l$.  
 The quantity $u^l\left( \tilde{c}\smile _{r+1}d^{\prime }|\tilde{d}={d},\tilde{c}\smile
_{r}d^{\prime }\right)$ is thus the conditional probability of the event $(\tilde{c}$ and $d'$ are nice at level $r+1$) given that they are nice at level $r$ and that the signal  received by player 2  is $d$. We divide the problem into different
cases.

\underline{Case $m>1$ and $r=2m-1$. }

Note that the events $\{\tilde{c}\smile _{2m}d'\}$ and $\{\tilde{c}\smile
_{2m-1}d'\}$ can be decomposed as follows :
\begin{eqnarray*}
\{\tilde{c} \smile_{2m-1}d'\}&=&\{\tilde{c}\smile _{2m-2}d'\}\cap
\{X_{d'_{m-1},\tilde{c}_{m}}=1\},\\
\{\tilde{c} \smile_{2m}d'\}&=&\{\tilde{c}\smile _{2m-2}d'\}\cap
\{X_{d_{m-1}^{\prime },\tilde{c}_{m}}=1\}\cap \{X_{\tilde{c}_{m},d'_m}=1\}.
\end{eqnarray*}%
So $u^l\left( \tilde{c}\smile _{2m}d^{\prime }|\tilde{d}={d},\tilde{c}\smile
_{2m-1}d^{\prime }\right)= u^l\left( X_{\tilde{c}_m,d'_m}=1| \tilde{d}={d},\tilde{c}\smile
_{2m-1}d^{\prime }\right)$, and the Markov property gives: 
\begin{eqnarray*}u^l\left( \tilde{c}\smile _{2m}d^{\prime }|\tilde{d}={d},\tilde{c}\smile
_{2m-1}d^{\prime }\right)& =& u^l\left( X_{\tilde{c}_m,d'_m}=1| X_{d'_{m-1},\tilde{c}_m}=1, X_{d_{m-1},\tilde{c}_m}=1, X_{\tilde{c}_m,d_m}=1\right),\\
& = & \frac{\sum_{i\in U} X_{i,d_{m}^{\prime }} X_{d'_{m-1},i} X_{d_{m-1},i}X_{i,d_{m}}}
{\sum_{i\in U}X_{d'_{m-1},i}X_{d_{m-1},i}X_{i,d%
_{m}}}.
\end{eqnarray*}%

This is equal  to 
$\frac{1}{2}\frac{Y_{d_{m},d_{m}^{\prime
}}^{d_{m-1},d_{m-1}^{\prime }}}{Y_{d_{m}}^{d%
_{m-1},d_{m-1}^{\prime }}}$ if $d'_{m-1}\neq d_{m-1}$, and to $\frac{1}{2}\frac{Y_{%
d_{m},d_{m}^{\prime }}^{d_{m-1}}}{Y_{d_{m}}^{d_{m-1}}%
}$ if  $d_{m-1}^{\prime }=d_{m-1}$. In both cases,  $E$ implies %
(\ref{eq63}).

\underline{Case $r=2m$.}

We have $u^l\left(\tilde{c}\smile _{2m+1}d'|\tilde{d}=d,\tilde{c}\smile _{2m}d^{\prime
}\right) =u^l\left( X_{d_{m}^{\prime },\tilde{c}_{m+1}}=1|\tilde{d}=d,\tilde{c}\smile _{2m}d^{\prime
} \right)$, and by the Markov property :
\begin{eqnarray*}
u^l\left(\tilde{c}\smile _{2m+1}d'|\tilde{d}=d,\tilde{c}\smile _{2m}d^{\prime
}\right) &=&u^l\left( X_{d_{m}^{\prime },\tilde{c}_{m+1}}=1| X_{{d_m},\tilde{c}_{m+1}}=1, X_{\tilde{c}_{m+1}, d_{m+1}}=1\right),\\
&=&\frac{\sum_{i\in U}X_{d_{m}^{\prime
},i}X_{d_{m},i}X_{i,d_{m+1}}}{\sum_{i\in U}X_{d_m,i}X_{i,d_{m+1}}} \\
&=&\frac{1}{2}\frac{Y_{%
d_{m+1}}^{d'_{m},d_{m}}}{Y_{d_{m+1}}^{d_{m}}}\;  \in [1/2-\alpha,1/2+\alpha]. \end{eqnarray*}%


\underline{Case $m=1$, $r=1$.}
\begin{eqnarray*}
u^l\left( \tilde{c}\smile _{2}d'|\tilde{d}=d,\tilde{c}\smile _{1}d^{\prime
}\right) &=&u^l\left( \tilde{c}\smile _{2}d'|\tilde{d}=d\right), \\
&=&u^l\left( X_{\tilde{c}_{1},d_{1}^{\prime }}=1|X_{\tilde{c}_{1},d_{1}}=1\right),\\
&= &\frac{\sum_{i\in U}X_{i,d'_{1}}X_{i,d_{1}}}{\sum_{i\in
U}X_{i,d_{1}}}\\
& =&\frac{1}{2} \frac{Y_{d_{1},d_{1}^{\prime }}}{Y_{d_{1}}}\;  \in [1/2-\alpha,1/2+\alpha]. 
\end{eqnarray*}

Let us now consider   condition  $UI1$:    we require that  for all $l\geq 1$,  
all ${c}=({c}_{1},...,{c}_{l})$ in $C^{l}$ and $c^{\prime
}=(c_{1}^{\prime },...,c_{l+1}^{\prime })$ in $C^{l+1}$  such that $%
{c}_{1}=c_{1}^{\prime }$, we have 
\begin{equation*} 
  u^l \left( c^{\prime }\smile _{2l+1}\tilde{d}\;  \big| \; \tilde{c}={c},c^{\prime
}\smile _{2l}\tilde{d}\right)  \in [1/2-\alpha,1/2+\alpha] \hspace{3cm} (\ref{eq61})
\end{equation*}%
and  for all $m\in \{1,...,l\}$ such that ${c}%
_{m}\neq c_{m}^{\prime }$, for $r=2m-2,2m-1$,
\begin{equation*} 
 u^l \left(c^{\prime }\smile _{r+1}\tilde{d} \; \big| \; \tilde{c}=c,c^{\prime
}\smile _{r} \tilde{d}\right)  \in [1/2-\alpha,1/2+\alpha].\hspace{3cm} (\ref{eq62})
\end{equation*}%

We start with (\ref{eq61}). 
 \begin{eqnarray*}
u ^{l}\left( c^{\prime }\smile _{2l+1}\tilde{d}|\tilde{c}=c,c^{\prime }\smile
_{2l}\tilde{d}\right) &=&u^l\left( X_{\tilde{d}_{l},c_{l+1}^{\prime
}}=1|\tilde{c}=c,c^{\prime }\smile
_{2l}\tilde{d}\right), \\
& = & u^l\left( X_{\tilde{d}_{l},c_{l+1}^{\prime
}}=1|X_{c_{l}^{},\tilde{d}_{l}}=1,X_{c_{l}^{\prime },\tilde{d}_{l}}=1\right) , \\
&=&\frac{\sum_{i\in V}X_{i,c_{l+1}^{\prime }} X_{c_{l},i}X_{c_{l}^{\prime
},i}}{\sum_{i\in V}X_{c_{l},i}X_{c_{l}^{\prime
},i}} .
\end{eqnarray*}
 This is $\frac{1}{2}\frac{Y_{c_{l+1}}^{c_{l},c_{l}^{\prime }}}{Y^{%
c_{l},c_{l}^{\prime }}}$ if $c_{l}^{\prime }\neq c_{l}$, and $\frac{1}{2}\frac{Y_{c_{l+1}}^{c_{l}}}{Y^{%
c_{l}}}$ if $c_{l}^{\prime }= c_{l}$. In both cases,  (\ref{eq61})
holds.\\

We finally consider (\ref{eq62}) and distinguish several case.

\underline{Case $r=2m-1$ and $m=l$.} 

\begin{eqnarray*}
u^{l}\left(c^{\prime}\smile _{2l}\tilde{d}|\tilde{c}=c,c^{\prime }\smile
_{2l-1}\tilde{d}\right) &=&u^{l}\left( X_{c_{l}^{\prime },\tilde{d}_{l}}=1| \tilde{c}=c,c^{\prime }\smile
_{2l-1}\tilde{d} \right) , \\
&=& u^{l}\left( X_{c_{l}^{\prime },\tilde{d}_{l}}=1|  X_{c_{l}^{},\tilde{d}_{l}}=1 \right) , \\
&=&\frac{\sum_{i\in V}X_{c'_{l},i}X_{c_{l},i}}{\sum_{i\in
V}X_{c_{l},i}},\\
&=& \frac{1}{2}\frac{Y^{c'_{l},c_{l}^{}}}{Y^{c%
_{l}}} \; \in [1/2-\alpha,1/2+\alpha].
\end{eqnarray*}%

\underline{Case $r=2m-1$ and $m<l$.} 

\begin{eqnarray*}
u^{l}\left(c^{\prime}\smile _{2m}\tilde{d}|\tilde{c}=c,c^{\prime }\smile
_{2m-1}\tilde{d}\right) &=&u^{l}\left( X_{c_{m}^{\prime },\tilde{d}_{m}}=1| \tilde{c}=c,c^{\prime }\smile
_{2m-1}\tilde{d} \right) , \\
&=& u^{l}\left( X_{c_{m}^{\prime },\tilde{d}_{m}}=1|  X_{c_{m},\tilde{d}_{m}}=1,  X_{\tilde{d}_{m}, c_{m+1}}=1 \right) , \\
&=& \frac{\sum_{i\in V}X_{c'_{m},i}X_{c_{m},i}X_{i,c_{m+1}}} {\sum_{i\in V}X_{c_{m},i}X_{i,c_{m+1}}},\\
&=& \frac{1}{2}\frac{Y^{c'_{m},c_{m}}_{c_{m+1}}} {Y^{c_{m}}_{c_{m+1}}} \; \in [1/2-\alpha,1/2+\alpha].
\end{eqnarray*}

\underline{Case $r=2m-2$.}

\begin{eqnarray*}
u^{l}\left(c^{\prime}\smile _{2m-1}\tilde{d}|\tilde{c}=c,c^{\prime }\smile
_{2m-2}\tilde{d}\right) &=&u^{l}\left( X_{\tilde{d}_{m-1},c_{m}^{\prime }}=1| \tilde{c}=c,c^{\prime }\smile
_{2m-1}\tilde{d} \right) , \\
&=& u^{l}\left( X_{\tilde{d}_{m-1},c_{m}^{\prime }}=1|  X_{c'_{m-1},\tilde{d}_{m-1}}= X_{c_{m-1},\tilde{d}_{m-1}}=  X_{\tilde{d}_{m-1}, c_{m}}=1 \right) , \\
&=& \frac{\sum_{i\in V}X_{i,c'_{m}}X_{i,c_{m}}X_{c'_{m-1},i}X_{c_{m-1},i}} {\sum_{i\in V}X_{i,c_{m}} X_{c'_{m-1},i}X_{c_{m-1},i}}.
\end{eqnarray*}
This is  $\frac{1}{2}\frac{Y^{c'_{m-1},c_{m-1}}_{c'_{m},c_m}} {Y^{c'_{m-1}c_{m-1}}_{c_{m}}}$ if $c_{m-1}\neq  c'_{m-1}$, and      $\frac{1}{2}\frac{Y^{c_{m-1}}_{c'_{m},c_m}} {Y^{c_{m-1}}_{c_{m}}}$  if $c_{m-1}=  c'_{m-1}$. In both cases, it belongs to $[1/2-\alpha,1/2+\alpha]$, concluding the proofs of lemma \ref{lemT4}, proposition \ref{pro3}   and theorem \ref{thm3}. \end{proof}

\bibliographystyle{plain}
\bibliography{zerosumnew}

\end{document}